\documentclass[12pt]{article}
\newcommand{\blind}{1}
\usepackage{graphicx}
\usepackage{amssymb}
\usepackage{natbib}
\usepackage{url}
\usepackage{multirow}
\usepackage {booktabs}
\usepackage[colorlinks,linkcolor=blue,anchorcolor=black,citecolor=blue,urlcolor=black]{hyperref}
\usepackage{amsthm, amsmath, amsfonts, amssymb,mathrsfs}
\theoremstyle{plain}
\newtheorem{theorem}{Theorem}
\newtheorem{lemma}{Lemma}

\theoremstyle{definition}
\newtheorem{assumption}{Condition}
\newtheorem{definition}{Definition}
\newtheorem{remark}{Remark}
\newtheorem{example}{Example}
\usepackage[plain,ruled,noend]{algorithm2e}
\usepackage{multirow, multicol}
\usepackage{float, subcaption}
\usepackage{bm, enumerate, xcolor}

\graphicspath{{figures/}}

\newcommand{\Set}{\tau}
\newcommand{\ms}{\mathcal{M}}
\newcommand{\km}{\mathrm{K}}
\newcommand{\E}{\mathbb{E}}
\newcommand{\pr}{\mathbb{P}}
\newcommand{\true}{\mathrm{true}}

\newcommand{\risk}{\mathscr{R}}
\newcommand{\Eff}{\mathrm{Eff}}
\newcommand{\mA}{\mathcal{A}}
\newcommand{\mB}{\mathcal{B}}
\newcommand{\npull}{\mathrm{N}_{\rm check}}

\newcommand{\select}{\psi}

\begin{document}
\def\spacingset#1{\renewcommand{\baselinestretch}%
{#1}\small\normalsize} \spacingset{1}

\if1\blind
{
  \title{\bf A reinforced learning approach to optimal design under model uncertainty}
  \author{Mingyao Ai$^1$, Holger Dette$^2$, Zhengfu Liu$^3$, and Jun Yu$^3$ \thanks{All the authors are equally contributed to this work. Authors listed alphabetically by their surname. The corresponding author is Jun Yu (yujunbeta@bit.edu.cn)} \\
  $1:$ LMAM, School of Mathematical Sciences and Center for Statistical Science,\\ Peking University\\
$2:$ Ruhr-Universit\"{a}t Bochum, Fakult\"{a}t f\"{u}r Mathematik\\
    $3:$ School of Mathematics and Statistics, Beijing Institute of Technology }
  \maketitle
} \fi

\if0\blind
{
  \bigskip
  \bigskip
  \bigskip
  \begin{center}
  {\LARGE\bf A reinforced learning approach to optimal design under model uncertainty}
\end{center}
  \medskip
} \fi

\bigskip

\bigskip
\begin{abstract}

Optimal designs are usually model-dependent and likely to be sub-optimal if the postulated model is not correctly specified. In practice, it is common that a researcher has a list of candidate models at hand and a design has to be found  that is efficient for selecting the true model among the competing  candidates  and is also efficient (optimal, if possible) for estimating 
the parameters of the true model.
In this article, we use a reinforced learning approach to address this problem. We develop  a  sequential algorithm, 
which generates a sequence of designs  which
have asymptotically, as the number of stages 
increases, the same efficiency for estimating the parameters in the true model as  an optimal design if the  true model would have correctly been specified in advance. 
A lower bound is established to quantify the relative efficiency between such a design and an optimal design for the true model in finite stages. 
Moreover, the resulting designs are also efficient for discriminating between the true model and other rival models
from the candidate list.    Some connections with other state-of-the-art algorithms for model discrimination and parameter estimation are discussed and the methodology is illustrated by a small simulation study.

\end{abstract}
\noindent%
\vfill

\newpage
\spacingset{1.4}\par
\section{Introduction}\label{sec:intro}
  \def\theequation{1.\arabic{equation}}	
	\setcounter{equation}{0}

To enhance the statistical efficiency of experiments with limited budgets, optimal design is one of the most commonly used concepts in statistics. Meanwhile, there exists an enormous amount of  literature on constructing optimal designs  \citep[see, for example,][and the references therein]{pukelsheim2006optimal,atkinson2007sas, Luc2013design}. 
Most of the existing approaches are model dependent, which means that the statistical relationship between control factors and possible responses is assumed to be known before planning an experiment. 
Therefore, one can leverage model information to find optimal designs.

In  many applications, a model cannot be specified before the experiment and a  common approach to address this situation in the construction of designs is  to assume that the  true model 
is contained  in a  class of candidate models. In this case,  a good design should be efficient for  the identification of the true model and for the estimation of the corresponding parameters. However, it  is well known that an  optimal design for model discrimination  often has a relatively poor performance for parameter estimation \citep[see, for example,][]{ATKINSON2008DT}. On the other hand, an optimal design for parameter estimation under a given parametric model usually does not provide enough design points to check its goodness-of-fit, for example by comparing it with a wider candidate model. 

Because of the importance to incorporate model uncertainty  in the design of experiments, several authors have worked on the problem of  constructing efficient designs for
  competing  candidate models and meanwhile there exists a vast amount of  literature on this topic. A common approach is to  conduct a two-stage (or multi-stage) design where the first stage is tailored for model discrimination, and 
in the later stages, an optimal design is constructed for the selected model \citep[see, for example,][]{Hill1968joint,montyeh1998}. 
Other authors propose  designs to optimize   compound or constraint optimality criteria that incorporate  all candidate models. 
Typical examples are the maximization average  or the minimum of  design efficiencies for all candidate models
 \citep[see][among many others] {Elisabeth1974Experimental,Dette1990A,wood2006designs}
 or the  maximization of a criterion for one particular model under the constraint that the efficiencies in all other candidate models exceed certain thresholds 
 \citep[see][among others]{cookwong1994,biedethof2009}.
Alternatively, one may embed the model of interest in a more general model and then find 
either an optimal design for estimating  all parameters in the enlarged model or an optimal 
design for estimating the parameters in the enlarged model  which are not contained in the original model
\citep[see][]{stigler1971,songwong1999}. Closely related to this approach are $T$-optimal designs, which are particularly tailored 
to model discrimination between not necessarily nested models  \citep{ATKINSON1975the,Lopez2007an,dette2009Optimal}. Moreover, some model-free designs, such as uniform designs, are usually served as another possible solution. In particular, 
\cite{Hickernell2002Uniform} demonstrate that uniform designs yield both reasonable efficiencies and robustness
and   \cite{Moon2012two} and \cite{joseph2015} point out that uniform designs with good projection properties are desirable for factor screening.
Finally, one may combine optimality criteria for model discrimination and parameter estimation by a compound or constrained criterion to find 
a good design for both  purposes
\citep[see][]{dette2000constrained,ATKINSON2008DT,TOMMASI2009optimal,Lopez2007an,May2013MODELSA}.
Another possibility is  to hybrid two types of designs \cite[see, for example][]{WATERHOUSE2008design}.

Despite the great success achieved by these methods,  they 
do not yield to an optimal solution, namely a  design, which approximates the optimal design for parameter estimation in the true (but before the experiment unknown) model under investigation. 
In particular, selection uncertainty unavoidably occurs in multi-stage approaches and the resulting designs
may be far away from the optimal one. 
Increasing budgets on model discrimination will mitigate selection inaccuracy, but the resultant designs may still have poor performance 
since discrimination designs are usually not efficient  for parameter estimation.
 Except for some specific examples, 
one can not expect that a robust design is also an optimal design for parameter estimation under the true model.
Optimal designs for  a compound or constraint criterion addressing  model discrimination and parameter estimation also suffer from a similar problem.

{\bf Our contribution: }
In this article, we use a reinforced learning approach to construct a sequence of sequential designs, which approximates the optimal design  for the true (but unknown) model in the presence of a list of competing candidate models.  
Our basic idea is to select in each stage  an optimal design from the list of  optimal designs for the candidate models  by  Thompson sampling,  which has been proved to be optimal in the reinforcement learning framework \citep[see, for example,][]{Auer2002finite,Agrawal2012anlysis}. The 
regret required for this approach is defined by design efficiencies, and  the posterior distribution is updated by a model selection step 
in each stage. As our method achieves a minimal regret in design efficiency, we obtain an  asymptotically optimal design (for the true model) combining all designs from the different  stages.    
It is proved that the sequence of  efficiencies (in the true model) of the  resulting designs converge to one, meaning  that the designs approximate  the optimal design in the true model. In particular, we derive explicit 
lower bounds for the efficiency of the sequential design after a finite number of 
stages.

By the  reinforced learning approach, we implement 
a design strategy, which  addresses  the different statistical objectives  at different stages of the experiment and 
can  be regarded as an efficient adaptive  hybrid design.
The  works most similar to ours are  \cite{Biswas2002An}, \cite{TOMMASI2009optimal}
 and \cite{May2013MODELSA}, who propose sequential algorithms in model discrimination and parameter estimation for nested linear and non-linear models.
Here the key idea is to sequentially adjust the weights in a compound (or hybrid) criterion that incorporates the estimation efficiency in the  different candidate models. However, these  methods only take the 
design in the final stage into account and 
rely heavily on  the assumption of nested 
 models. Moreover, both the sequential  designs in \cite{Biswas2002An} and  \cite{May2013MODELSA} require that all competing models are estimable using the data from the  current stage  or all historical data. As a result,  the experimenter may spend more costs in model discrimination than  is really necessary.  The reinforced learning approach avoids these problems and provides an interesting alternative to the commonly used strategies for designing experiments under model uncertainty.


\section{Preliminaries}\label{sec:preliminaries}

  \def\theequation{2.\arabic{equation}}	
	\setcounter{equation}{0}
 
\subsection{Problem setups}

Consider an experiment, where for a given predictor, say $\bm x$,  a response $y$ can be observed. For the statistical analysis, it is usually assumed that $y$ is a realization of a random variable $Y$ with a conditional distribution function
$F(y|\bm x)$. Let ${\cal X}$ denote  a compact subset of an Euclidean space, then we assume 
that we can take observations at each $\bm x \in {\cal X}$, where observations at different 
experimental conditions are assumed to be independent. 
As pointed out in the introduction, a large part of the literature on optimal design assumes a parametric model, say $F(y|\bm x;\bm\beta )$  for the distribution of $Y$.
Here  $F$ is a known conditional distribution 
and  $\bm\beta$ is the parameter of interest, which has to be estimated from the data.  

In  many applications, it is difficult to fix a concrete model
before any experiments have been carried out. However, often
the experimenter has several candidate models in mind which could fit the data well.
To address this problem in the construction of optimal designs, we assume  that 
the ``true'' model is an element of a set of  $\km$ candidate models denoted by $\ms=\{M_1,  \ldots, M_{\km} \}$,
where the distribution of the $j$th model $M_j$ is given by 
$F_j(y|\bm x;\bm\beta_j)$ and 
$\bm\beta_j$ denotes the 
$p_j$-dimensional parameter of interest  ($j=1, \ldots  , \km$).
We emphasize that  the design space for the predictor $\bm x$ may be different for different models although this is not reflected in our notation.
In this work, we aim to find  designs that 
have the highest estimation efficiency under the true model.

We begin with introducing some basic concepts and notations 
of optimal design theory  for a given  model, say  $M_j \in \mathcal{M} $. 
We define an 
approximate design as a probability measure 
with finite support, that is 
$\xi=\{(\bm x_i,\omega_i):i=1,\ldots,m\}$,
where $\bm x_1, \ldots  , \bm x_m  \in {\cal X} $, 
$\omega_1   , \ldots , \omega_m >0$  and $\sum_{i=1}^m\omega_i=1$.  If $n$ independent observations can be taken to estimate $\bm  \beta_j$, the quantities $n \omega_i$ are rounded to integers, say $n_i$, 
such that $\sum_{i=1}^m n_i = n$,
and $n_i $ observations are taken at each $\bm x_i$ ($i=1, \ldots , m$). In this case, under standard assumptions, the asymptotic covariance 
matrix of the maximum likelihood estimator 
is given by  $n^{-1} I_j^{-1}(\xi  )$, where
$I_j(\xi)=\sum_{i=1}^m\omega_i I_j(\bm x_i, \bm \beta_j )$
is the  information matrix of the design 
$\xi$, $I_j(\bm x_i, \bm \beta_j )$ denotes the Fisher information matrix at a single design point $\bm x_i$ (in model $M_j$), and we do not reflect a possible dependence of the information matrix $I_j(\xi)$ on the parameter $\bm \beta_j$ in our notation. {The model $M_j$ 
(more precisely the parameter  $\bm \beta_j$) is called estimable  by the design $\xi$ if 
the matrix $I_j(\xi)$ is non-singular.}
An optimal design for the model $M_j$ maximizes a real-valued function, say $\phi_j$ of the 
information matrix $I_j(\xi)$
in the class of all designs. 
Here $\phi_j$ is an information function in the sense of \cite{pukelsheim2006optimal}, that is a positively homogeneous, concave, non-negative, non-constant, and upper semi-continuous function on the space of non-negative definite matrices.
Typical examples include 
the famous $D$-optimality criterion $\phi_j (\xi) = \{\det (I_j(\xi))\}^{1/p_j}$  and the
$A$-optimality criterion
$\phi_j (\xi) =  \{{\rm tr}( {p_j}^{-1}  I_j^{-1}(\xi) )\}^{-1}$,  where 
$p_j$ denotes the dimension of the parameter $\bm \beta_j$. Throughout this paper we denote a design maximizing $\phi_j$ as $\phi_j$-optimal design and 
 assume that  $\phi_j (\xi) =0 $ if the information matrix $I_j(\xi) $
is singular. This property is 
satisfied for almost all optimality criteria considered in practice.

In many cases, in particular for non-linear models,  the information matrix also depends on the unknown parameter.
 Following   \cite{chernoff1953locally}, 
we use  the concept of 
``locally optimal'' design  and
assume in this case that  an initial guess
for the unknown parameter $\bm\beta_j$ is available. 
Such an approach is reasonable if the 
locally optimal designs are not too sensitive with respect to  the choice of the unknown parameter. A typical  example
appears in the modeling of dose-response relationships 
in phase II clinical trials by non-linear regression models, where already some information from phase I is available \citep[see, for example,][]{bretz2006}.  
In this context, \cite{DetteBretz2010}
investigated the robustness of optimal designs for estimating the minimum effective dose
from a  dose-response curve 
with respect to the
specification of the unknown parameter. Their results suggest that locally optimal designs are moderately robust with respect to a misspecification of the model parameters but highly sensitive with respect to a misspecification of the regression function.

To address this problem 
in the design of experiments under model uncertainty, 
we denote  the index  corresponding to the true model, say $M_\true \in \ms = \{M_1,  \ldots, M_{\km} \}$,  by  the subscript  ``$\true$'' in the following discussion. 
Of course,  an optimal design for the true model in the  set  $\ms$ conceptually refers to a design $\xi^*$ maximizing 
$ \phi_{\true}(\xi) = \phi_{\true}(I_\true (\xi)) $. However, a little care is necessary since the optimal design for the true model may  not be unique in general, and an optimal design under an incorrect model may also be optimal for the true model
as indicated in the following example.

\begin{example} ~~

\noindent
Assume that the class $\ms$ consists of three
trigonometric regression models
$$
Y=\beta_0+\sum_{l=1}^j\left(\beta_{2l-1}\cos(lt)+\beta_{2l}\sin(lt)\right)+\varepsilon~,~~
$$
of degree  $j=1,2,3$ and that the design space is given by the interval $ [0,2\pi)$.
It is proved in Section 9.16 of \cite{pukelsheim2006optimal} that every 
uniform design  with more than $2j+1$ support points  is   $\Phi_q$-optimal for the trigonometric regression model of degree $j$, 
which means that the design  maximizes the criterion 
$\Phi_q(\xi)=\{{p_j}^{-1}{\rm tr}(I_j^{q} (\xi) )\}^{1/q}$, where $q \in [-\infty,1)$.  Note that the cases $q=0$, $q=-1$, and $q=-\infty$ correspond to widely used $D$-, $A$- and $E$-optimality criterion, respectively. Consequently, a uniform design 
with seven equispaced support points  is optimal for all three models. 
However, when $M_1$ or $M_2$ is  the  true model, one can also construct a design with less than seven supports to achieve the same efficiency. In this case, an experimenter would prefer this design  since it further reduces the cost of altering the experimental settings.  
\end{example}

In the following, we formally define an optimal design for the true model among the competing candidate models from the set $\ms$.

 \begin{definition}\label{def1}
 	A design $\xi^*$  is said to be $\ms$-optimal, if and only if 
 	$\xi$ maximizes the functional 
  \begin{equation} \label{eq1}
 		 \phi_{\true}(\xi) \gamma_{\xi},
 \end{equation}
    where $\gamma_{\xi}$  is an indicator, which is  one if and  only if the number of support points of $\xi $  is 
    minimal among all  $\phi_\true$-optimal  
   designs for the true model.
 \end{definition}

Note that  the true  model can not be estimated by the
design under consideration if the corresponding  Fisher information matrix $I_\true (\xi)$ is singular. In this case, 
the criterion value is  zero (because the criteria vanish for 
singular matrices, by assumption). Therefore,  maximizing  the criterion in \eqref{eq1}  will always result in a design that can be used to estimate all parameters in the true model.
On the other hand, a design constructed for a more complicated  model is not necessarily  a good solution, since a more complex model usually involves more support points. In this scenario, statistical efficiency is not the only standard, the economic costs  caused by taking observations at more different experimental conditions  
 should also  be taken into account.
In general $\phi_{\true}(\xi) \gamma_{\xi}$ will not be maximal  unless  the design $\xi$ is $\phi_{\true}$-optimal under the true model
$M_\true$ with a  minimal number of support
points. 

It is worth mentioning that constructing an optimal design with a minimal number of  support points is not an easy task. In practice, this condition can be relaxed by the optimal design with a minimal number of support points among optimal designs for the candidate models. 
For the sake of a simple notation, all  optimal designs in this paper refer to the optimal designs for the corresponding models with  a minimal number of  support points.

 \subsection{{Multi-stage} designs} 
 
 One can not expect to obtain  an $\ms$-optimal design when the experiment can only be carried out in one stage without enough prior knowledge.
In the following, we introduce  multi-stage (approximate) designs which is one of the key ingredients in finding an $\ms$-optimal design.
 
In the multi-stage design context,   $n_t$  observations  are taken according to  an approximate design  $\xi_{(t)}$ at each   stage $t$ for $t\in\{1, \ldots ,T\}$, where $T$ denotes the  total number of possible stages.
 Let $n:= \sum_{t=1}^{T} {n_t}$
 be the total sample size.
We are interested in  finding an $\ms$-optimal multi-stage design, that is    a sequence of designs $\xi_{(1)},\ldots, \xi_{(T)}$ such that the convex combination  $ \xi_{\ms ,T} = 
\sum_{t=1}^{T}(n_t/n)\xi_{(t)}$
is in some sense ``close''  to an 
$\ms$-optimal design.
For this purpose  we define a corresponding
 gain of conducting $\xi_{(t)}$ at   stage $t$
 by  
\begin{equation}\label{eq:reward}
	r(\xi_{(t)})= \phi_{\true}(\xi_{(t)}) \gamma_{\xi_{(t)}},
\end{equation}
where $\gamma_{\xi_{(t)}}$ is defined 
in Definition~\ref{def1}. 
Our approach is then  based on 
finding a sequence of designs $\xi_{(1)}, \ldots, \xi_{(T)}$ that achieves the maximum gain at each stage $t$.
One possible solution to  reach this goal  is to select a design from the locally $\phi_1, \ldots, \phi_K$-optimal designs 
$\xi_1^*,\ldots,\xi_{\km}^*$ for the models $M_1, \ldots , M_{\km}$, respectively.
Clearly, if one can select the design $\xi_{\true}^*$ infinitely often, 
such that  ${\# \{  t   | \xi_{(t)} \not =  \xi_{\true}^* \} }/T$
converges to $0$ as $T \to \infty$ {and the   sample sizes 
$n_t$ at each stage are of the same order}, then an asymptotically  $\ms$-optimal design can be constructed by aggregating these designs. By this approach, we simplify the problem to a discrete choice problem.
 Finding a maximum gain design  in each stage could then serve as a proxy for finding an $\ms$-optimal design.

\section{Sequential designs by reinforced learning}\label{sec:optimal_design}
  \def\theequation{3.\arabic{equation}}	
	\setcounter{equation}{0}

To achieve a maximum gain in each stage, the researchers 
have  to appropriately balance two objectives:   (a) exploiting the most promising model in the current stage and finding an optimal design to further improve its parameter estimation; 
(b) exploring an alternative model that may turn out to be the true model in the future and performing an optimal design for it.

Borrowing ideas from reinforcement learning, we will to use  multi-armed bandits' strategies to achieve the  balance between exploitation and exploration.
To be precise, one can treat the $\km$ candidate designs as arms. One may conduct a batch of experiments sequentially trying out various arms (or candidate designs in our setting) and collect gains to prioritize a design that is good for the true model.
Identifying the arm with the highest expected gain and pulling  this arm (i.e., conducting experiments according to the corresponding optimal design) infinitely often could then serve as a natural solution for constructing a design which is approximating an $\ms$-optimal design.

As pointed out in \cite{agrawal2013further},  Thompson's sampling strategy  \citep{Thompson1933ONTL} is a Bayesian heuristic reinforcement learning algorithm that can achieve a nearly optimal solution
for the dilemma between exploitation and exploration. 
In the following, we will link the Thompson sampling strategy 
to the multi-stage $\ms$-optimal design. 
Similar to traditional Thompson sampling, the proposed algorithm for {searching a} multi-stage $\ms$-optimal design alternates among the selection, evaluation, and updating.

We start with some necessary notations.
Denote the observations at the stage $t$
according to  a design $\xi_{({t})}$ by
\begin{equation}\label{eq:Dt}
    {\cal D}_{t}:= \{(\bm x_{n_0 + \ldots +  n_{t-1}+1}, Y_{n_0 + \ldots +  n_{t-1}+1}),\ldots, (\bm x_{n_0  + \ldots + n_{t}},Y_{n_0  + \ldots + n_{t}})\}
\end{equation}
and define 
 $\zeta_j(t)\in\{0,1\}$ as an indicator whether the model $M_j$ is chosen as the true model  using the data ${\cal D}_{t}$. It is worth mentioning that $\zeta_j(t)$ is a random variable.
Through the lens of
Bayesian model checking and selection techniques, we can further define the posterior probability 
\begin{equation}
\label{det11} 
\theta_{j}(\xi_{({t})}):= 
\pr(\zeta_{j}(t)=1|{\cal D}_{t})
\end{equation}
that $M_j$ is the true model conditional on the data ${\cal D}_t$. For mathematical rigor, we only define $\theta_{j}(\xi_{(t)})$ for the models which are  estimable.
As in  traditional Thompson Sampling, the $t$th  stage starts off by sampling  $\eta_j(t)\sim\mathrm{Beta}(a_j(t),b_j(t))$  from a posterior distribution that serves as a surrogate for 
$\theta_{j}(\xi_{(t)})$. To be precise,  $a_j(t)-1$ and $b_j(t)-1$
are the numbers of times that model $M_j$ is treated as the true model (success) and the improper model (failure) in the first $t-1$ stages, respectively ($j=1, \ldots , \km$).  
Then,  following the Thompson sampling policy we  select the index $\Set_t \in \{ 1, \ldots , \km  \} $  corresponding to the highest value among $\eta_1 (t) , \ldots, \eta_\km (t)$  at  stage  $t$ and  denote  by   $\xi_{\Set_t}^*$ the corresponding optimal design. For the first stage, we simply put $a_j(1)=b_j(1)=1$ for all candidate models opportunities to be selected as the true model at the beginning.

Then in the evaluation step, the user will 
evaluate all models  in the candidate pool and score them 
using a sample of observations from  the design 
 $\xi_{\Set_t}^*$. However, some care is necessary here, as there might exist models which are not estimable by the design  $\xi_{\Set_t}^*$.
To address this problem we conduct some additional experiments  at stage $t$ according to a uniform design. 
More precisely, in each stage $t$, we take $n_t$ observations according to a hybrid design $\rho_t\xi_{\Set_{t}}^*+(1-\rho_t)\xi_{\rm unif}$, where $\xi_{\rm unif}$ denotes a uniform design and $\rho_t\in(0,1)$ denotes the proportion of  observations taken according to  the optimal design at stage $t$.
We  use the uniform design here since it has desirable maximin properties with respect to goodness-of-fit testing \citep{Wiens2009RobustDD}. Other discrimination designs can also be applied here as well.
We will  allow the proportion $\rho_t$ to change with $t$, since as the observations began to accumulate we have more confidence in the optimal design we selected. A possible specification of $\rho_t$ is  $N_{\Set_{t}}(t)/(N_{\Set_{t}}(t)+1)$, where $N_j(t)$ denotes  the number of times that design $\xi_j^*$ was  chosen  until stage $t$.
Because all models are estimable by this design  we can use a model selection criterion to select the ``best''  model, say $M_{\select_t}$, among the candidates. In the following we will work with the  Bayesian information criterion  \citep[BIC, see][]{schwarz1978Estimating}, but other 
model selection techniques including goodness-of-fit tests  could be considered as well.
Finally, we assign a score $\zeta_{\select_t}(t)=1$ for the selected   model and the score $\zeta_{j}{(t)}=0$ for $j\neq {\select_t}$.

The algorithm concludes with an  {updating step} for the beta posterior distribution, setting $a_j(t+1)=a_j(t)+\zeta_j(t)$ and $b_j(t+1)=b_j(t)+(1-\zeta_j(t))$. The details are 
summarized in  Algorithm  \ref{alg:vanillathompson1}.
\medskip

\begin{algorithm}[H]
{\SetAlgoLined
\spacingset{1.2}
\small
\LinesNumbered
\medskip 

{\bf Initialize} {$n_0:=0$ and} 
$a_j(1)=1$ and $b_j(1)=1$ for each arm $j=1,\ldots,\km$.
\medskip

\For{$t=1,\ldots,T-1$}{
	
\textbf{(1) Selection Step}
	\begin{itemize}
	    \item Generate random variables $\eta_j(t)\sim\mathrm{Beta}(a_j(t),b_j(t))$ for $j=1,\ldots,\km$.
     \item Choose the one with the largest value among $\{\eta_j(t)\}_{j=1}^\km$ and denote its index by $\Set_t$.
     \item {Take $n_t$ observations 
${\cal D}_t$ defined in \eqref{eq:Dt}
according to  $\xi_{(t)}^{\rm hyb}:= \rho_t\xi^*_{\Set_t}+(1-\rho_t)\xi_{\rm unif}$,\\ where $\xi^*_{\Set_t}$ is the (locally) optimal design 
   for  model $M_{\Set_t}$}, and $\xi_{\rm unif}$ is a uniform design.
 
	\end{itemize}

\textbf{(2) Evaluation Step }

    \begin{itemize}
        \item Use the data ${\cal D}_t$ and  the BIC criterion to  select 
        a model in  $\mathcal{M}$, and denote the corresponding index by 
        $\select_t$. 
        \item $\zeta_j(t) = 1$ (true model) if   $j=\select_t$  and $\zeta_j(t) = 0 $ if   $j \not =\select_t$. 
    \end{itemize}

\textbf{(3)  Updating Step}
 \begin{itemize}
	\item	Set $a_j(t+1)=a_j(t)+\zeta_j(t)$ and $b_j(t+1)=b_j(t)+(1-\zeta_j(t))$. 
  \end{itemize}
}
\medskip

{\bf Output:} The  design ($n:= \sum_{t=1}^{T} {n_t}$)
\begin{equation}
  \label{result} 
  \textstyle{\xi_{\ms ,T} : =\sum\limits _{t=1}^T \frac{n_t}{n}\xi_{(t)}^{\rm hyb}= \sum\limits _{t=1}^T \frac{n_t}{n}\left\{\rho_t\xi^*_{\Set_t}+(1-\rho_t){\xi_{\rm unif} 
} \right\}.}
\end{equation}
}

\caption{\small Multi-stage $\ms$-Optimal Design Algorithm} 
\label{alg:vanillathompson1}
\end{algorithm}

\begin{remark}

Constructing an optimal design in the selection step is not an easy task
and in many cases these designs have to be found numerically. Fortunately, 
many efficient numerical algorithms  have been developed  for generating optimal designs in  different models  
with respect to different criteria \citep[see, for example,][for some recent references]{yu2010Monotonic,yang2013algorithm,R_Software_OptimalDesign}.
Thus, we do not specify a detailed algorithm  here for  the selection step.
After obtaining an approximate optimal design, users may apply a rounding procedure \citep{pukelsheim2006optimal} to decide about the number of  replications at each design point. Alternatively, the exact optimal design algorithm based on mixed-integer second-order cone programming  proposed by  \cite{Guillaume2015Computing} can also be applied.
\end{remark}

\begin{remark}\label{remark:testset}
    To  reduce  the costs, we can first conduct $n_{\rm test}$ experiments
    at the beginning.
    Then we use these $n_{\rm test}$ 
    observations together with $n_t - n_{\rm test} $ 
    observations according to the optimal design $\xi_{\Set_{t}}^*$ at each stage $t$ 
    to select the best model.
    Note  that conditional on these $n_{\rm test}$ experimental units, the selection results are still independent thus the effects of choosing improper designs will also not be propagated over time.
\end{remark}

\begin{remark}\label{remark:testset1}
In some cases (for example, if the sample size $n_t$ is too small) not all models in the class $\ms $  may be  estimable using observations  according to the 
    hybrid design $ \xi_{(t)}^{\rm hyb}:= \rho_t\xi_{\Set_{t}}^*+(1-\rho_t)\xi_{\rm unif}$ (or by the design in Remark~\ref{remark:testset}). 
    In such a scenario, we propose to modify the operation at stage $t$ as follows. 
We skip the updating step (3) for the models which are not estimable, say $M_j$, and keep $a_j(t+1)=a_j(t), b_j(t+1)=b_j(t)$. 
Consequently, we still have chance to explore such models at the stage $t+1$ due to the relatively large variance of  the $\mathrm{Beta}(a_j(t+1),b_j(t+1))$ distribution.
 For all   estimable models, we adopt a goodness-of-fit test \citep[see, for example][]{dette1999} and  assign $\zeta_j(t)=0$  to all  models which are rejected. Among the models  that cannot be reject by the goodness-of-fit test, we 
 select a model by BIC  and define  $\zeta_{\select_t} (t)=1$ for the selected model and $\zeta_j(t)=0$
  for the rest. Note that  this procedure will assign 
 $\zeta_{\select_t} (t)=0$ to all models of $\ms$ if the goodness-of-fit test rejects all models.
\end{remark}

\begin{remark}
    Compared with searching compound or constrained optimal designs
 our method offers a computational advantage   since we do not need to search for  designs for various values of  weights or constraints in the criteria.
    This is a considerable benefit when the multi-objective design is hard to obtain.
\end{remark}

\begin{remark}
{Unlike most sequential designs such as \cite{Biswas2002An}, which use  all  observations from previous stages to conduct the model evaluation, Algorithm \ref{alg:vanillathompson1} only uses the data $\mathcal{D}_{t}$
from  stage $t$. As a consequence, the algorithm has some ability for exploration and the effects of choosing   improper designs will not be propagated over time.} 

\end{remark}

\begin{remark}
	From the evaluation step and updating step in Algorithm~\ref{alg:vanillathompson1}, one can clearly see the difference to the single-armed bandit problem.
	In the evaluation step, we propose to evaluate as many models as possible to make full use of the information from an experiment.
	Thus, sometimes it can be more efficient than the single-armed bandit in the sense of data exploration.
	More importantly, such an evaluation method can help the researcher  to  efficiently identify the true model from the candidates that are well-fitted (or overfitted) with the observed data.

\end{remark}

\section{Theoretical results}\label{sec:thm}
  \def\theequation{4.\arabic{equation}}	
	\setcounter{equation}{0}
 
In this section, we investigate some theoretical properties of the proposed algorithm. 
In particular, two fundamental questions regarding  Algorithm~\ref{alg:vanillathompson1} will be answered.
First, can the true model be correctly identified by the proposed design?
Second, how well approximates the design  $ \xi_{\ms ,T} $ obtained by Algorithm~\ref{alg:vanillathompson1} the optimal design $\xi^*_{\true}$ in terms of design efficiency?  For the sake of a simple notation, we assume throughout this section that the number of
runs in all stages is the same, i.e. $n_1=\ldots=n_T.$

\subsection{Asymptotic properties}
\label{sec41}

By Definition~\ref{def1}, the resulting design $ \xi_{\ms ,T} $  in \eqref{result} approximates  the  $\ms$-optimal design $\xi_{\ms}$  if one chooses $\xi^*_{\true}$ infinitely often in the selection step of Algorithm~\ref{alg:vanillathompson1} {such  that  $\#\{t|\xi_{\Set_t} \not = \xi^*_{\true}\}/T\to  0$.}
If this happens, the answers to the above two questions are straightforward.
In fact, we will show  below that the proposed method  selects a wrong design  only $O(\log(T))$ times - see   Lemma  \ref{lem:times} in the supplement, for more details. For notational simplicity, we assume $\rho_t=N_{\Set_{t}}(t)/(N_{\Set_{t}}(t)+1)$ as suggested in Section~\ref{sec:optimal_design}.
 Before we present the main results  we will impose a mild condition on the  selection criterion  used in the evaluation step (which is in our case the  BIC procedure).

\begin{assumption}\label{ass:stability}
 There exists two constants $\alpha\in(0,1/2)$ and $c\in (\alpha,1-\alpha)$, such that
any design $(1-\rho)\xi_l^*+\rho\xi_{\rm unif}$ with $\rho\in
{\cal R}:= \{1/2,2/3,3/4, \ldots\}$ and $l\in\{1,\ldots,\km\}$ satisfies   
\begin{equation}
  \label{cond11}  
\theta_{\true}(\rho\xi_l^*+(1-\rho)\xi_{\rm unif})>c+\alpha,
\end{equation}
 and 
\begin{equation}
  \label{cond12}  
\theta_{j}(\rho\xi_l^*+(1-\rho)\xi_{\rm unif})<c-\alpha,
\end{equation} 
for any $j\neq \true$.
\end{assumption}

Condition~\ref{ass:stability} essentially requires that the decision (or the selection) is stable according to different kinds of designs.
Note that the BIC enjoys the selection consistency property \citep[see, for example][]{claeskensHjort08} and  with probability converging  to one, this method is able to
select the true model from the candidate models.
Thus one can expect that Condition~\ref{ass:stability} holds when $n_t$ is not too small. In this case   we can show that the design obtained by Algorithm \ref{alg:vanillathompson1} approximates the 
$\ms$-optimal design  since the event $\{M_{\Set_t} \not = M_\true \}$ occurs only finitely many times as $T\to\infty$. Our first result makes this statement precise.
Note that by  definition the index   $\Set_t$ is a random variable  and consequently  the resulting design
$\xi_{\ms ,T}$  in Algorithm~\ref{alg:vanillathompson1}  
and its efficiencies  are random objects as well.

 \begin{theorem}\label{thm:optimal}
 	Assume that Condition~\ref{ass:stability}
  holds. Then, as $T\to\infty$, 
 	\begin{equation}
 		\phi_{\true}(\xi_{\ms,T})\to 	\phi_{\true}(\xi_{\true}^*)
 	\end{equation}
  in probability, 
 	where $\xi_{\ms,T}$ is the design obtained by Algorithm~\ref{alg:vanillathompson1}, and $\xi_{\true}^*$ is the (locally) optimal design for the true model.
 \end{theorem}

    Note that Condition~\ref{ass:stability} is  an  assumption on the model selection procedure and Theorem~\ref{thm:optimal} holds for any  procedure  that identifies the true model with high probability. 
    The BIC in Algorithm~\ref{alg:vanillathompson1} is one option  and one  can use other  techniques  as well to score the   candidate models.

Recall the notation  of  $N_j(t)$ as   the number of stages where  $\xi_j^*$ is chosen to construct the hybrid design in the selection step of 
Algorithm~\ref{alg:vanillathompson1} until time $t$. When the algorithm stops at time $T$, we propose  to finally  choose model $M_j$ such that 
$N_j(T)$ is the maximum number among $\{N_1(T), \ldots , N_\km (T) \}$. 
Our next result yields, as a by-product, the selection consistency of this method.

 \begin{theorem}\label{thm:selection}
If Condition~\ref{ass:stability} holds, then  the expected number of times that the event $\mA(t)=\{M_{\Set_t} \neq M_\true\}$ happens 
satisfies
$$
	\E \Big [ \sum_{t=1}^T \mathbb{I}\left(\mathcal{A}(t)\right) \Big ] 
  \le \km\left(2+\frac{2}{\alpha^2}\right)+\frac{8\log(T)}{\alpha^2}.
$$	
 In particular, as $T\to\infty$,  with probability approaching one, 
 $$
 N_\true (T) = \max \big \{N_1(T),\ldots,N_{\km}(T) \big \}~. 
 $$
\end{theorem}

\subsection{{Finite stage} analysis}
\label{sec42} 

 In many applications  {experimental units} are  expensive,  and the number of stages may be limited when users  have to take costs into account.
In the following, we will take a closer look at the potential loss when the number of different stages has to be restricted. For this purpose we
denote by  $\xi_{\true}^*$ the optimal design  for the true model and  define 
$$
\Eff_{\true}(\xi)= \frac{ \phi_\true(\xi)}{\phi_\true(\xi_{\true}^*)}
\in[0,1]
$$
as the $\phi_\true$-efficiency of a  design $\xi$.
 The following 
result gives a lower bound for the expectation 
of this random variable.

\begin{theorem}\label{thm:risk22}
	Assume that Condition~\ref{ass:stability} holds  
 and  that all models in ${\cal M}$ are estimable by  $\xi_{(t)}^{\rm hyb} (t=1,\ldots, T)$,
 then
	\begin{equation}
		\E \big [ \Eff_{\true}(\xi_{\ms,T}) \big ]  \ge 1- \frac{(8+\alpha^2)\log(T)}{\alpha^2T}-\frac{\left(2+{2}/{\alpha^2}\right)\km}{T}.
	\end{equation}
\end{theorem}

In the following,  we  provide an analogous  result
to Theorem  \ref{thm:risk22} 
when the Algorithm \ref{alg:vanillathompson1} is modified as  proposed in Remark~\ref{remark:testset1}.
For this statement, we require a slightly stronger assumption  than Condition \ref{ass:stability}.

\begin{assumption}\label{ass:stability1}
 If the true model is estimable by the design 
$\xi_{(t)}^{\rm hyb}$
at stage $t$, 
then  there exists two constants $\alpha\in(0,1/2)$ and $c\in (\alpha,1-\alpha)$, such that
 any design $(1-\rho)\xi_l^*+\rho\xi_{\rm unif}$ with  $\rho\in
{\cal R}:= \{1/2,2/3,3/4 , \ldots\}$ and $l\in\{1,\ldots,\km\}$ satisfies   
$$
\theta_{\true}(\rho\xi_l^*+(1-\rho)\xi_{\rm unif})>c+\alpha,
$$ and 
$$
\theta_{j}(\rho\xi_l^*+(1-\rho)\xi_{\rm unif})<c-\alpha,
$$ for any $j\neq \true$.
Otherwise, for any estimable
 model $M_j$, it holds that
$$
\theta_{j}(\rho\xi_l^*+(1-\rho)\xi_{\rm unif})<c-\alpha.
$$
\end{assumption}
Equipped with Condition~\ref{ass:stability1}, we are ready to provide a lower bound of $\E [ \Eff_{\true}(\xi_{\ms,T})
      ]$, which also directly implies the analogous  result
of  Theorem  \ref{thm:optimal} for the modification of  Algorithm  \ref{alg:vanillathompson1}  according to Remark \ref{remark:testset1}. 

\begin{theorem}\label{thm:risk21}
  	Assume that Condition~\ref{ass:stability1} holds  
    for all $t \in \{1,\ldots, T\}$. 
If the design $\xi_{\ms,T}$ is calculated by  
the modification of 
Algorithm \ref{alg:vanillathompson1}  as  proposed in Remark~\ref{remark:testset1}, then 
  		\begin{equation}
  			\E \big [ \Eff_{\true}(\xi_{\ms,T})
     \big ] \ge 1-\frac{(8\km-8+\alpha^2)\log(T)}{\alpha^2T}-\frac{\left(3+{4}/{\alpha^2}\right)\km}{T}-\frac{1}{T}.
  		\end{equation}
\end{theorem}

\begin{remark}
It is interesting to see that the risk bound in Theorem~\ref{thm:risk21} is  smaller compared with the bound in Theorem~\ref{thm:risk22}, especially for a large $\km$.
This is the price we  have pay for the cases where  one cannot explore all candidate models simultaneously.

As shown in \cite{LAI1985asymptotic}, the term $\log(T)$ is the leading order term in the lower bound for any multi-armed bandit algorithm. 
This can be regarded as a necessary price we pay for model exploration. 
Roughly speaking, in our algorithm the random variables  $\eta_j(t)$ may not be concentrated around the mode of the  $\mathrm{Beta}(a_j(t),b_j(t))$ distribution when $a_j(t)+b_j(t) $ is much smaller than $O(\log(T))$.

\end{remark}

\section{Numerical studies}\label{sec:sim}
  \def\theequation{5.\arabic{equation}}	
	\setcounter{equation}{0}
 
In this section  we study the finite sample properties of  the proposed method through several examples.

\subsection{Nonlinear models for dose-response studies}\label{ex:3}
In our first example we consider the following three non-linear models which are widely adopted in the dose-response experiments.  The three candidates are the Emax, linear and exponential model defined by 
\begin{equation}
\label{mod1}
\begin{split}
	M_1 :& ~~Y=
 f_1(d,\bm\beta_1)+\varepsilon:=
 \beta_0 +\beta_1 \frac{d}{d+\beta_2}+\varepsilon,\\
	M_2 :&~~  Y=
 f_2(d,\bm\beta_2)+\varepsilon:=
 \beta_0 +\beta_1 d +\varepsilon,\\
	M_3 :&~~ Y=
 f_3(d,\bm\beta_3)+\varepsilon:=
 \beta_0 +\beta_1 \exp(d/\beta_2)+\varepsilon,
 \end{split}
\end{equation}
respectively, where we assume that the error is standard normal distributed.
Following the setting in Section 10.5.3 of \cite{pinheiro_analysis_2006}, the dose $d$ varies in the interval  $[0,1]$ and the noise is normal distributed  with standard deviation $\sigma=0.65$. 
The vector of parameters is given by $\bm\beta_1= (0.2,0.7\delta,0.2)^\top$, $\bm\beta_2= (0.2,0.6\delta)^\top$ and $\bm\beta_3= (0,0.2,1/\log(1+3\delta))^\top$, for model  $M_1$, $M_2$ and $M_3$, respectively.
Here the parameter $\delta$ is used to control the signal-to-noise ratio (SNR) ${\rm var}(\E(Y|d))/\sigma^2$.
Note that  $d_\infty:=\min_{i\neq j}\sup_d|f_i(d,\bm\beta_i)-f_j(d,\bm\beta_j)|$  is given by   $0.5$ and $0.9$ for  $\delta=3$  and  $\delta=5$, respectively. We set that total sample size to be $n=150$.

In this example, we investigate  the $D$-optimality criterion, that is 
$\phi_j (\xi ) =  \big ( \det(I_{j}(\xi))\big )^{1/p_j}$. Consequently, the $D$-efficiency of a design for the true model is given by
\begin{equation}\label{det12}
  \Eff_{\true}^D(\xi) =
    \Big ( \frac{\det(I_{\true}(\xi))}{\det(I_{\true}(\xi_{\true}^*))
    } \Big )^{1/p_{\true}} ~, 
\end{equation}
 where $p_{\true}$ is the dimension of  the  vector of  parameters for the true model. 
 The locally $D$-optimal designs for the Emax ($M_1$) and exponential model ($M_3$) have been determined explicitly  in \cite{dette2010optimal}. They are uniform designs supported $d=0$ and $d=1$ and a third point in the interval $(0, 1)$, which depends on the 
 model under consideration. These authors also demonstrated that a design for a given model has often low $D$-efficiency for another model.  The optimal design for the linear model ($M_2$) is is  supported at $0$ and $1$ with equal weights .

We apply  Algorithm~\ref{alg:vanillathompson1}  with $T={10}$ stages and sample size $n_t= 15$ for $t=1, \ldots , {10}$  to conduct a sequential design.
Additionally, a ``standard design'' adopted in \cite{pinheiro_analysis_2006}
 with equal weights at the points $\{0,0.05,0.2,0.6,1\}$, 
a uniform design with five support points, a compound design
maximizing   $\{\Eff_{M_1}^D(\xi)\Eff_{M_2}^D(\xi)\Eff_{M_3}^D(\xi)\}^{1/3}$ and a hybrid optimal design defined by   $\sum_{k=1}^33^{-1}\xi_{k}^*$ are   considered for the sake comparison. 
These  designs can be found  in Section \ref{seca5} of  online supplement. 	
Besides the efficiency, the designs are evaluated by their selection accuracy (ACC), which is defined as  the probability  that 
 the true model is selected  using the BIC for model selection based on the final $150$ observations
 (which are obtained by the different designs).
 The efficiency  and   selection accuracy are estimated   by  $500$  simulation runs. 
In each simulation we use the $150$ observations 
obtained by Algorithm~\ref{alg:vanillathompson1} for model selection and parameter estimation. 
 For  the other  designs  considered in our study we simply take  $150$ 
new observations  at each replication
according to the design under consideration.

\begin{figure}[t!]
\begin{subfigure}[b]{0.3\columnwidth}
			\centering
   	\caption{$M_1$
   }
    	\includegraphics[width=1.0\columnwidth]{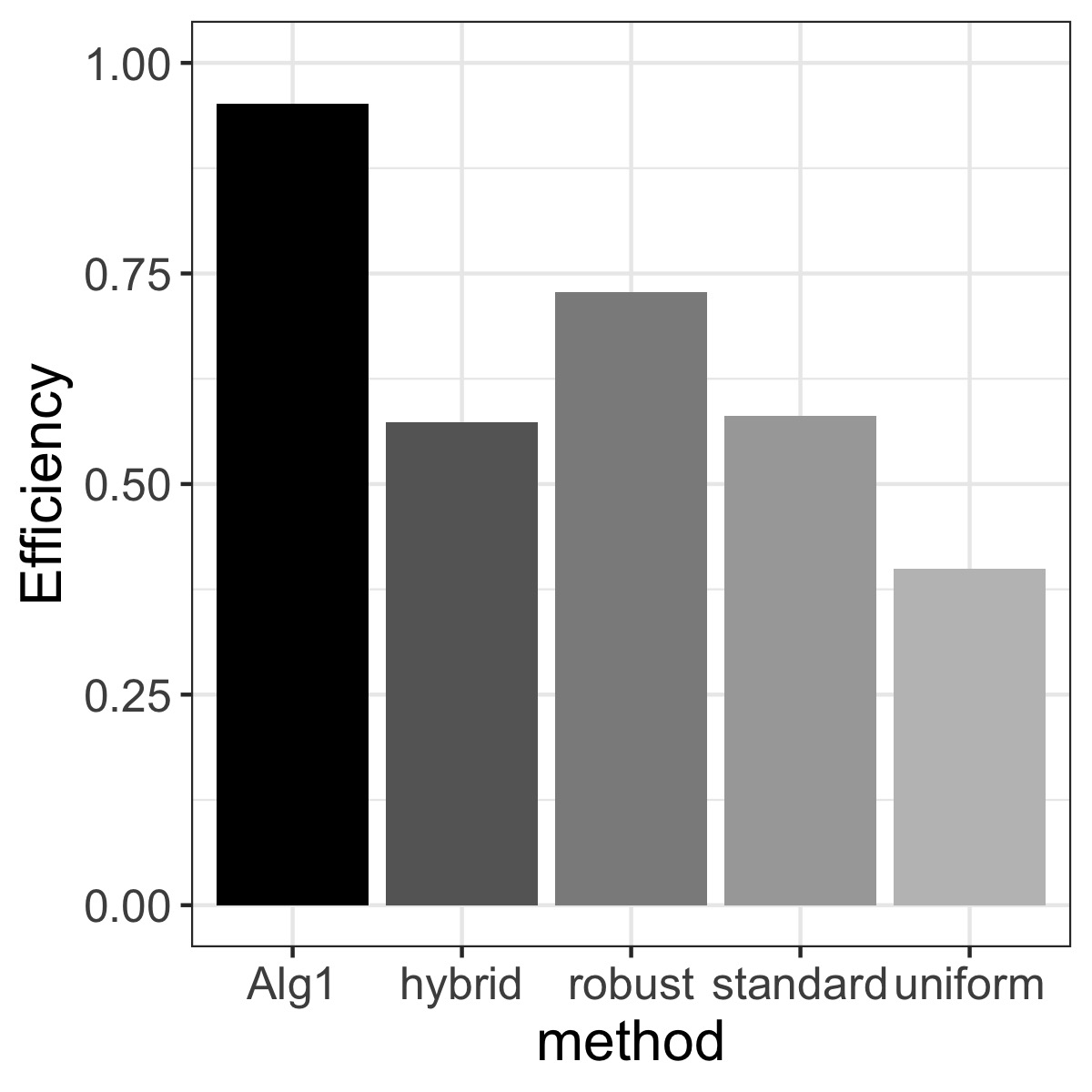}
		\end{subfigure}
		\begin{subfigure}[b]{0.3\columnwidth}
			\centering
   			\caption{$M_2$
   }
			\includegraphics[width=1.0\columnwidth]{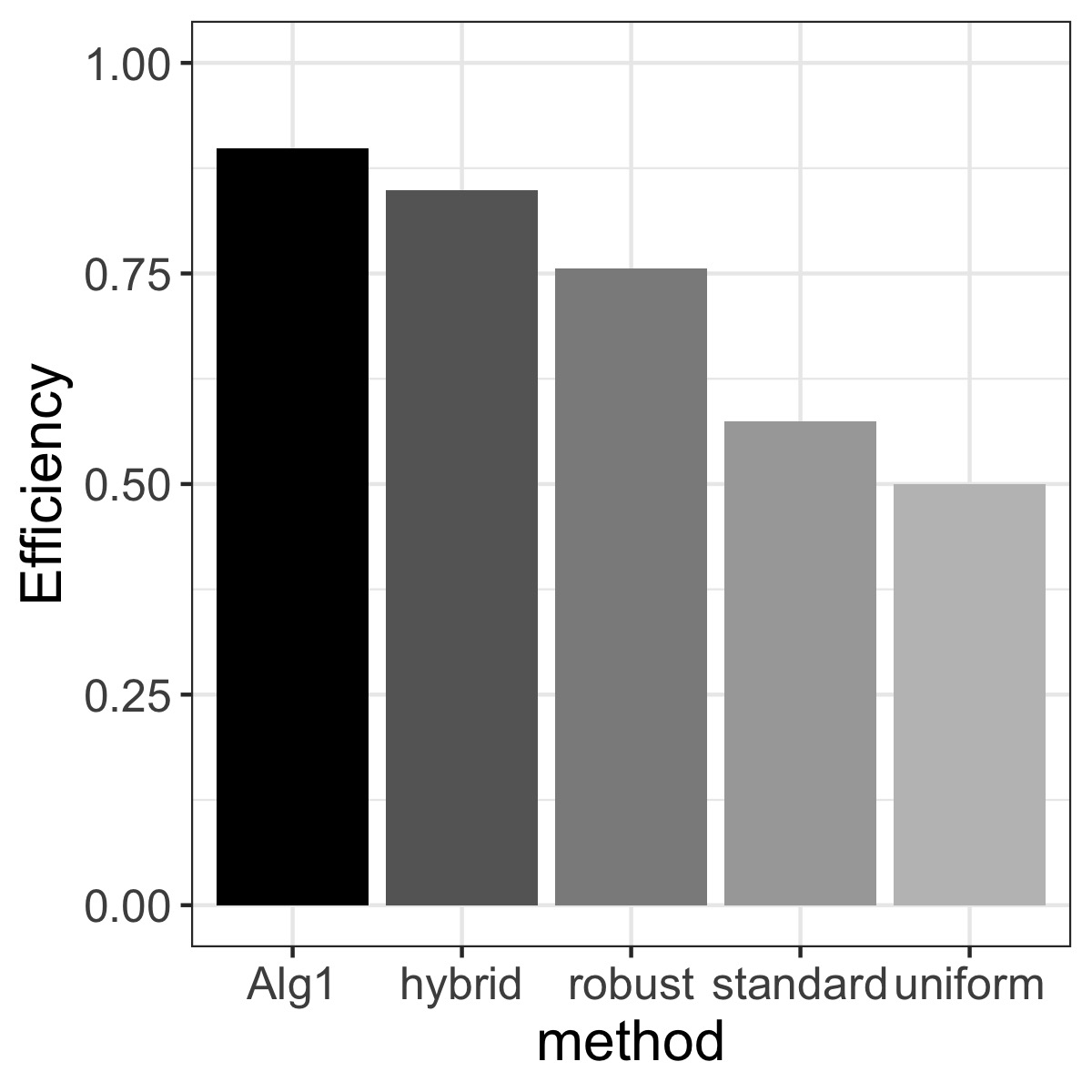}
		\end{subfigure}
		\begin{subfigure}[b]{0.3\columnwidth}
			\centering
   		\caption{$M_3$ 
   }
			\includegraphics[width=1.0\columnwidth]{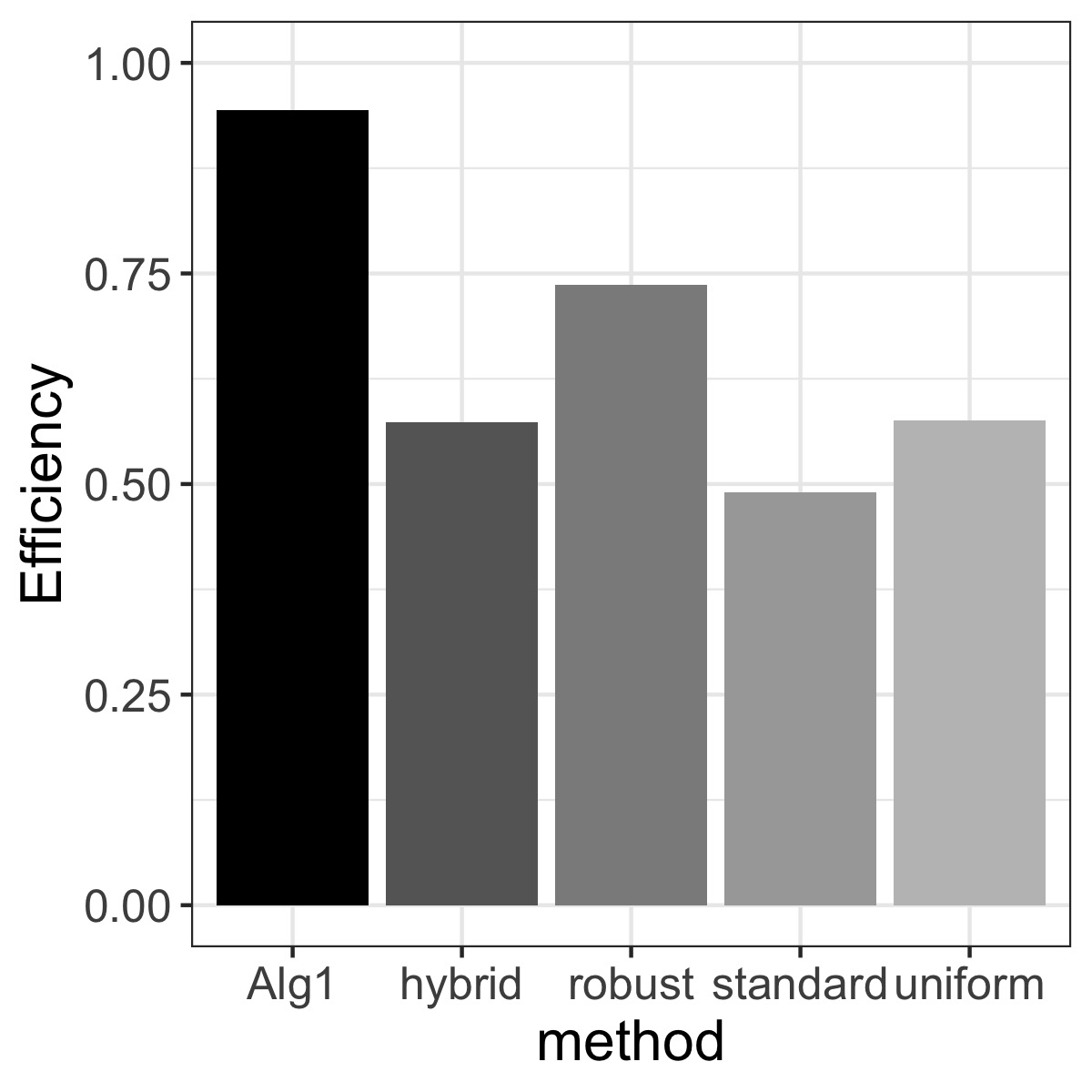}
		\end{subfigure}\\
  \begin{subfigure}[b]{0.3\columnwidth}
			\centering
   		\caption{$M_1$
   }
			\includegraphics[width=1.0\columnwidth]{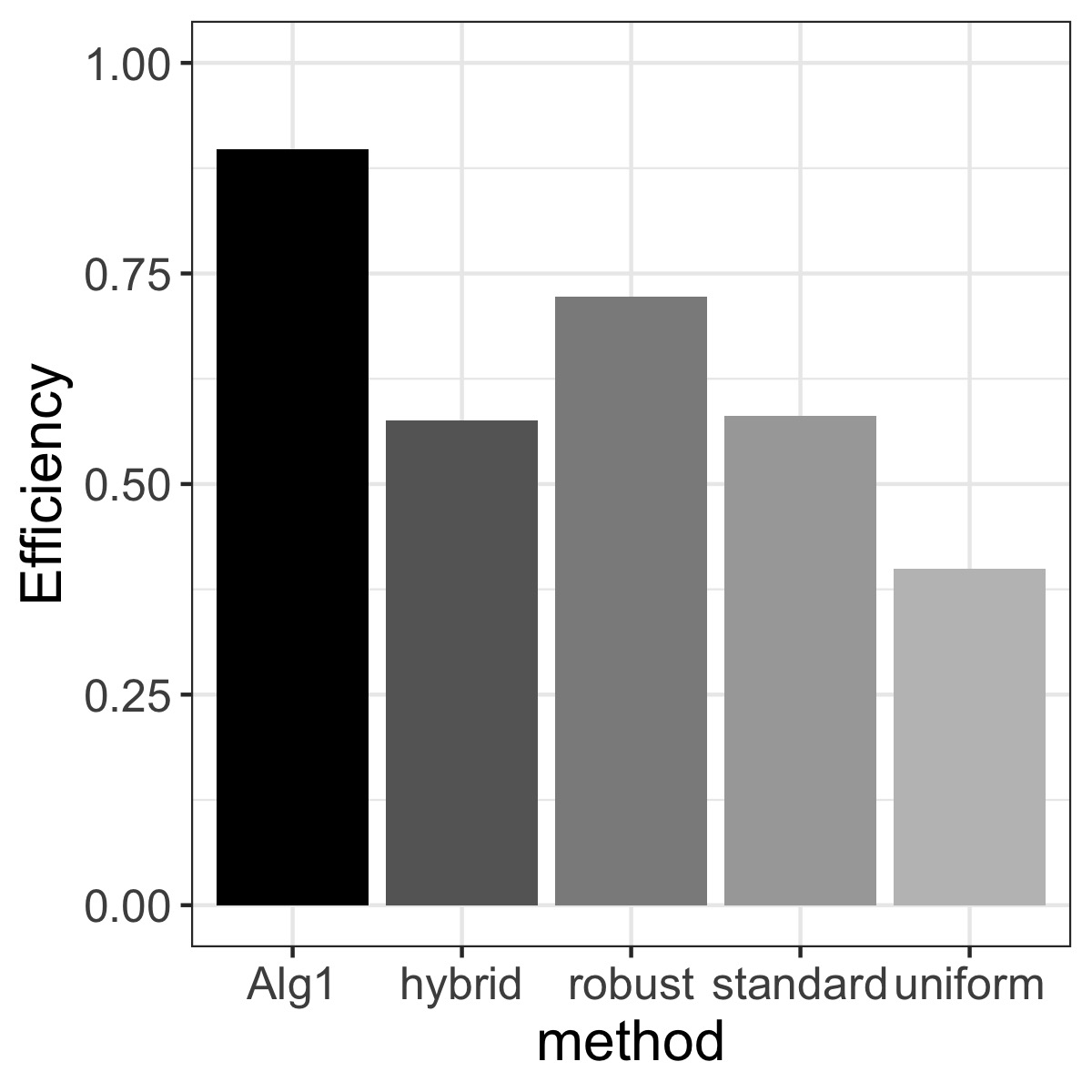}
		\end{subfigure}
		\begin{subfigure}[b]{0.3\columnwidth}
			\centering
   			\caption{$M_2$ 
   }
			\includegraphics[width=1.0\columnwidth]{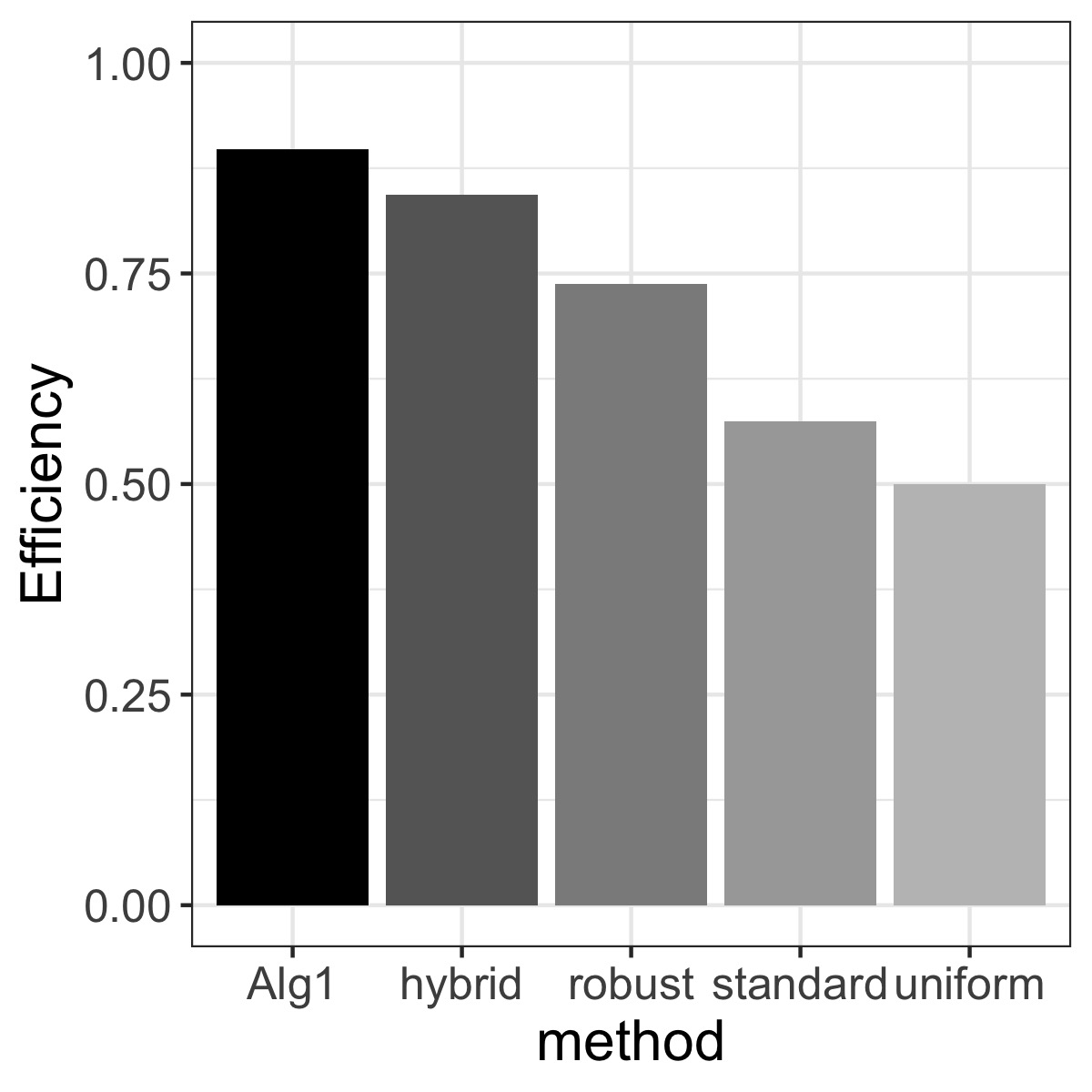}
		\end{subfigure}
		\begin{subfigure}[b]{0.3\columnwidth}
			\centering
   		\caption{$M_3$ 
   }
			\includegraphics[width=1.0\columnwidth]{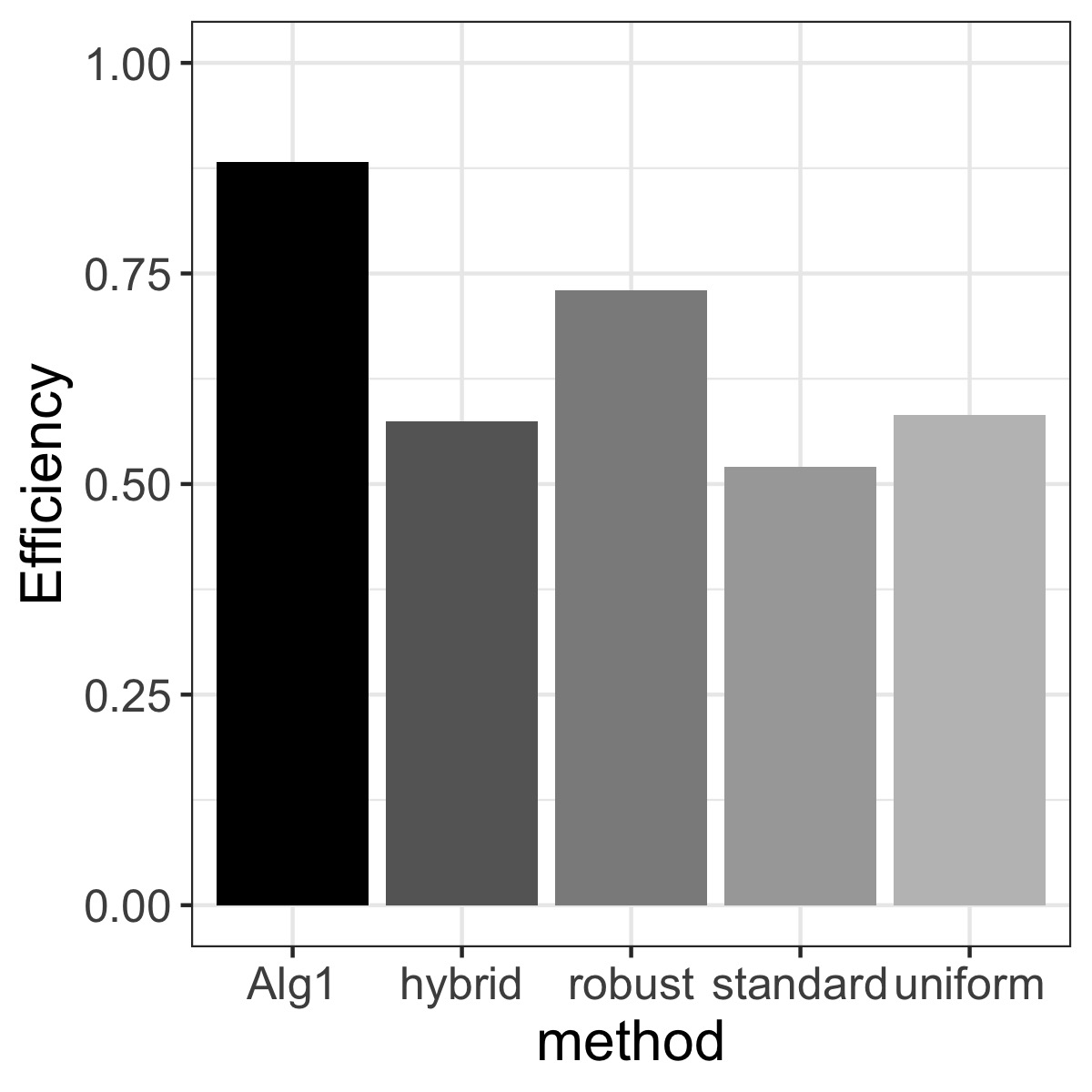}
		\end{subfigure}
		\caption{\label{fig:case3-barplot-eff}{
 \it 
 $D$-efficiencies defined in \eqref{det12} of different designs  under different true models defined in \eqref{mod1}.
 Upper part: SNR=3.75;  lower part: SNR=1.35.
  }}
		
	\end{figure}

The $D$-efficiencies are displayed in Figure~\ref{fig:case3-barplot-eff}.
We observe that for the SNR of $3.75$
the efficiencies  of 
the design  $\xi_{\ms.T}$ calculated by Algorithm \ref{alg:vanillathompson1}   are  close to $1$ in all cases under consideration.
 On the other hand, for  a  SNR  of $1.35$, the $D$-efficiency of the sequential designs are a little  smaller. Note that  in this case it is harder to distinguish the dose-response curves among the three models based on $15$ observations, which implies a smallr value of  $\alpha$ in Condition~\ref{ass:stability}.
 Thus the numerical results for  both scenarios   confirm our theoretical findings in Theorem~\ref{thm:optimal}  and \ref{thm:risk22}. Moreover, in all cases   under consideration the $D$-efficiencies of the proposed sequential  designs are larger than the efficiencies of the other competing designs under consideration.

 In Figure \ref{fig:case3}, we investigate 
the relation between the  $D$-efficiency   and total  sample size $n$ in case, where 
 $M_1, M_2$, and $M_3$ is  the true model and 
 the SNR is either $1.35$ or $3.75$. 
We  observe that    the 
 $D$-efficiencies quickly increase  to $1$  when the number of stages 
 increases. We observe that for a SNR of $1.35$  the $D$-efficiencies of the sequential  designs obtained by Algorithm~\ref{alg:vanillathompson1}  are a little smaller  than for $3.75$. The difference difference is mostly visible for model $M_3$. It can be explained 
 by the fact that in this case the linear and the exponential model 
 are very similar and therefore  the BIC tends to select a simpler model instead of the  true model. 
 
 \begin{figure}[H]
\begin{subfigure}[b]{0.3\columnwidth}
			\centering
   \caption{$M_1$}
   \vspace{-.2cm}
    	\includegraphics[width=1.0\columnwidth]{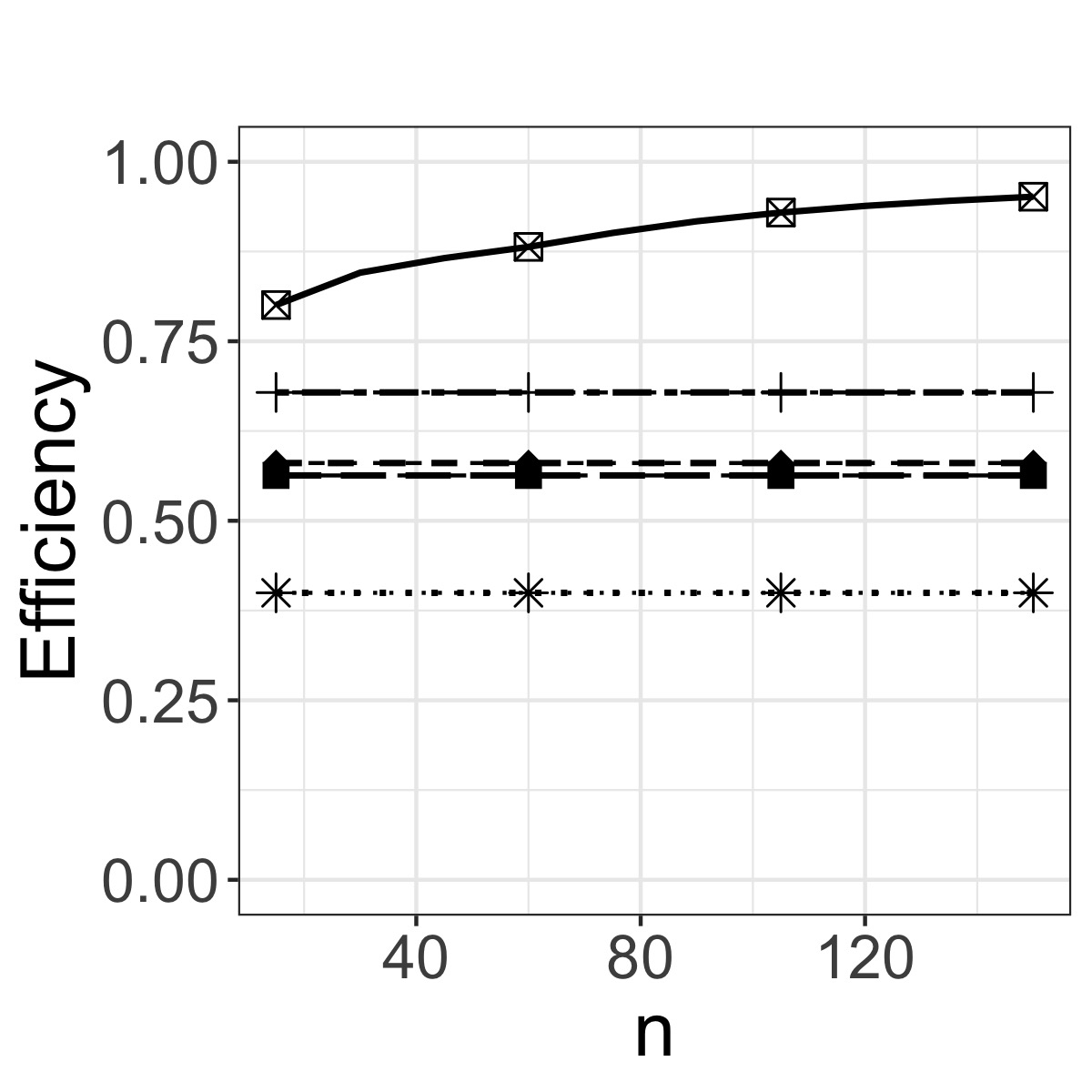}
		\end{subfigure}
		\begin{subfigure}[b]{0.3\columnwidth}
			\centering
     \caption{$M_2$}
        \vspace{-.2cm}
			\includegraphics[width=1.0\columnwidth]{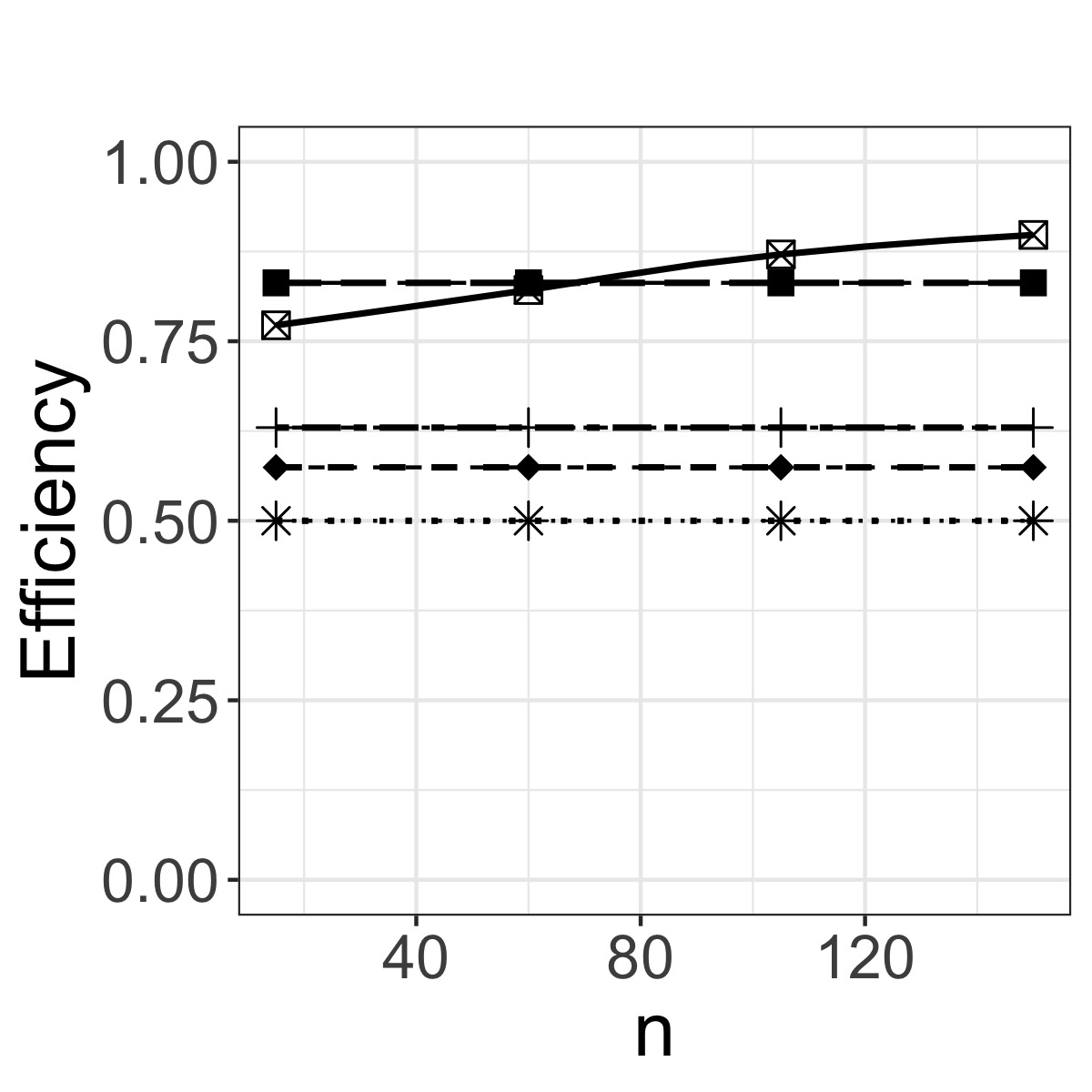}
		\end{subfigure}
		\begin{subfigure}[b]{0.3\columnwidth}
			\centering
     \caption{$M_3$}
        \vspace{-.2cm}
			\includegraphics[width=1.0\columnwidth]{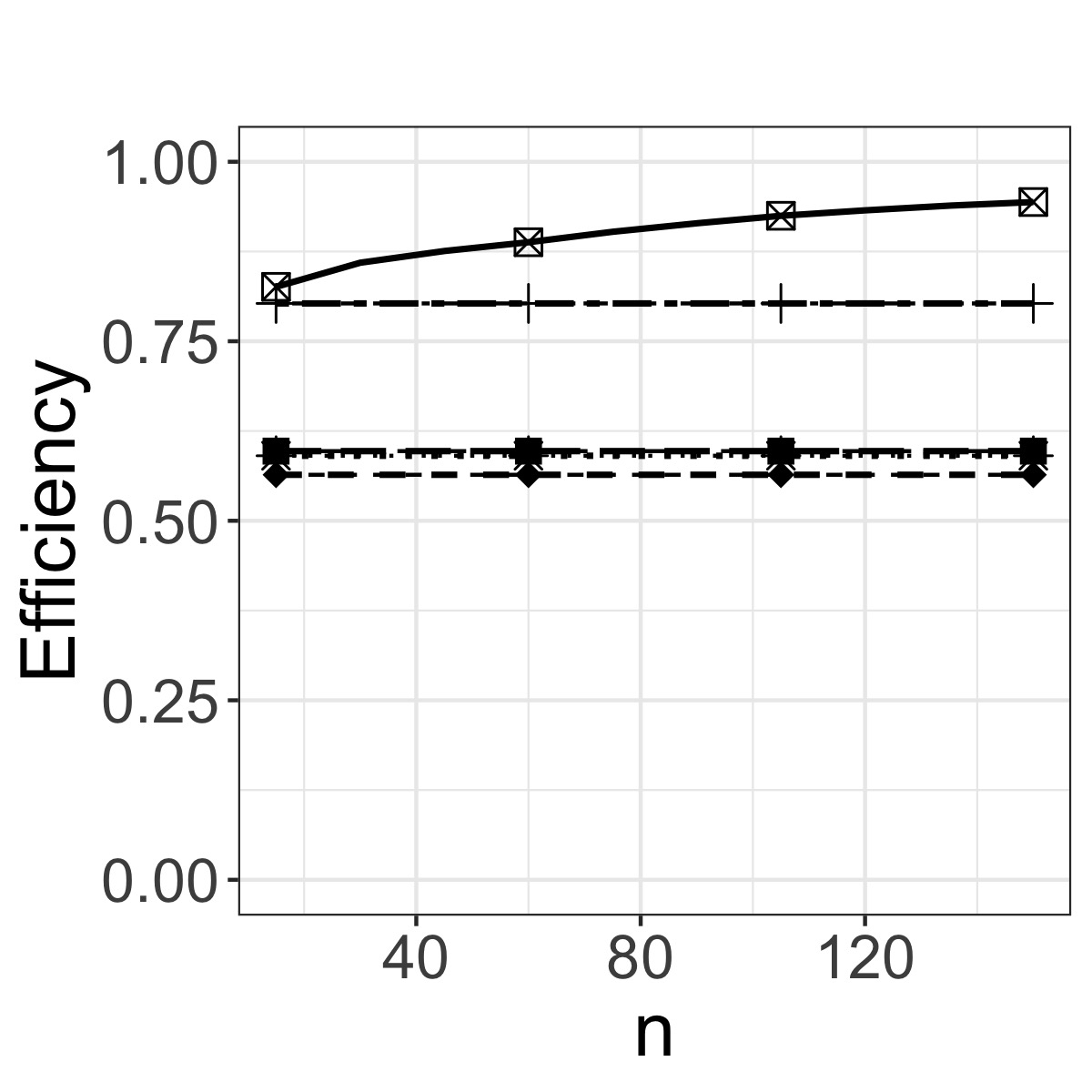}
		\end{subfigure}\\
  \begin{subfigure}[b]{0.3\columnwidth}
			\centering
       \caption{$M_1$}
        \vspace{-.2cm}
			\includegraphics[width=1.0\columnwidth]{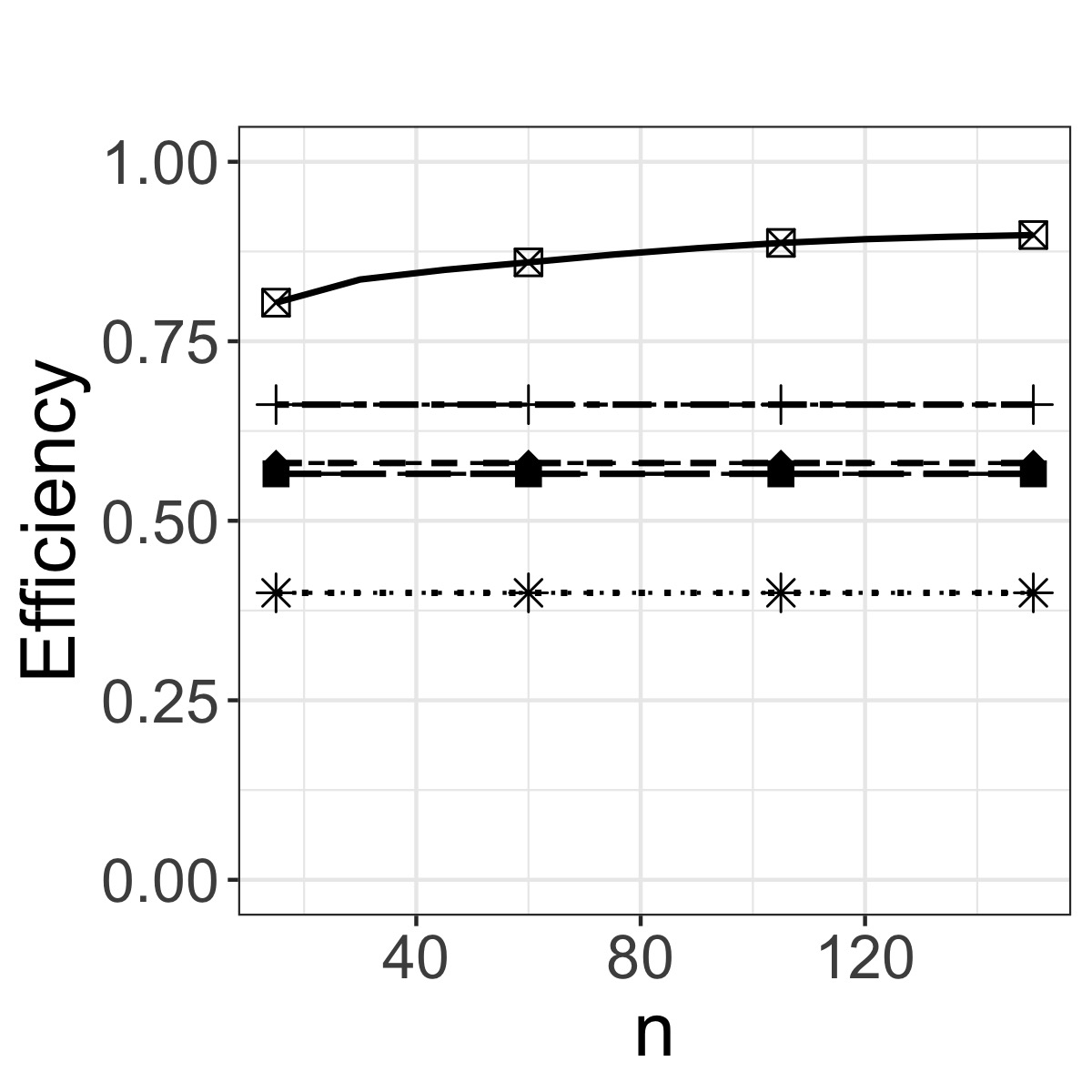}
		\end{subfigure}
		\begin{subfigure}[b]{0.3\columnwidth}
			\centering
       \caption{$M_2$}
        \vspace{-.2cm}
			\includegraphics[width=1.0\columnwidth]{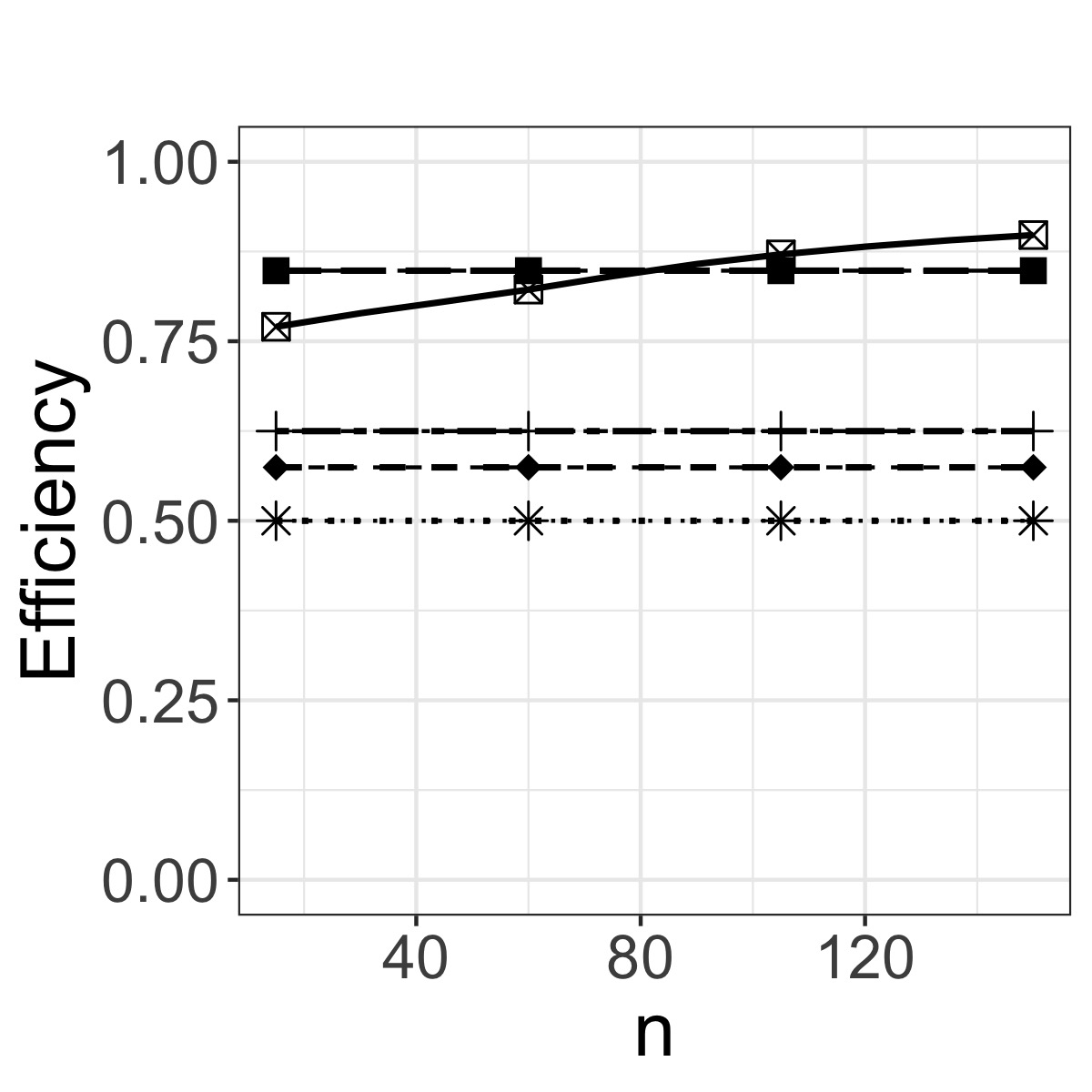}
		\end{subfigure}
		\begin{subfigure}[b]{0.3\columnwidth}
			\centering
       \caption{$M_3$}
        \vspace{-.2cm}
			\includegraphics[width=1.0\columnwidth]{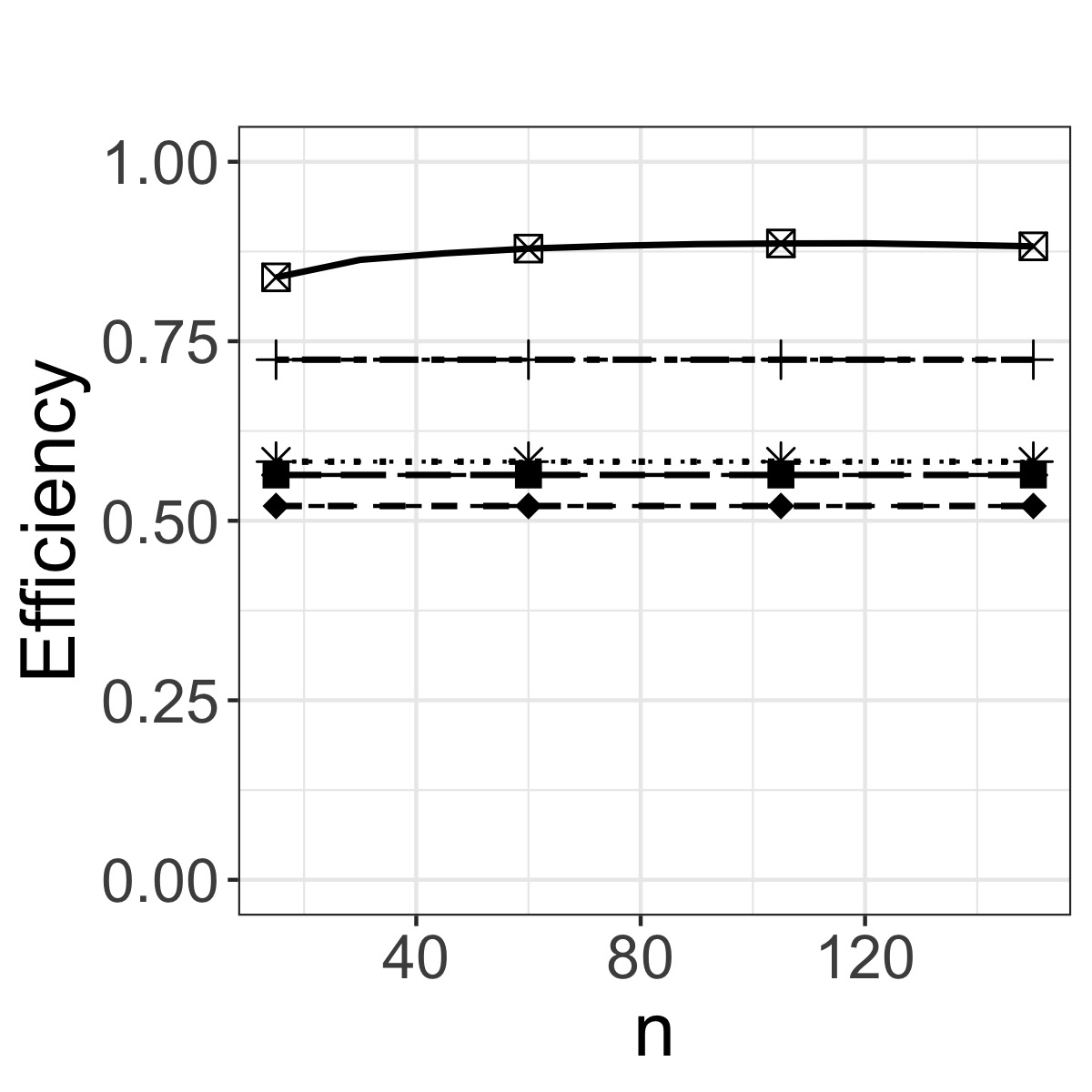}
		\end{subfigure}
		\caption{\label{fig:case3}{
 \it 
 $D$-efficiencies \eqref{det12} for the true model 
  as  the the total sample size $n$  of the sequential design is increasing. 
  The models are given by \eqref{mod1}.
  Five designs are compared: sequential design constructed by Algorithm \ref{alg:vanillathompson1} ($\boxtimes$);  robust design ($+$);  uniform design with five supports  ($\ast$); standard design  ($\blacklozenge$);  hybrid design ($\blacksquare$).
  Upper part: SNR=3.75;  lower part: SNR=1.35. 
  }}	
	\end{figure}
 
The corresponding selection accuracy of the BIC based on $150$ observations from the different designs is displayed in Figure \ref{fig:case3-barplot-acc}, and we observe that the differences between the different designs are rather small, except in the case where $M_3$ is the true model and the SNR is $1.35$. Here the uniform design and the design obtained by Algorithm \ref{alg:vanillathompson1} yield the best accuracy.

\begin{figure}[t!]
\begin{subfigure}[b]{0.3\columnwidth}
			\centering
   \caption{ $M_1$}
    	\includegraphics[width=1.0\columnwidth]{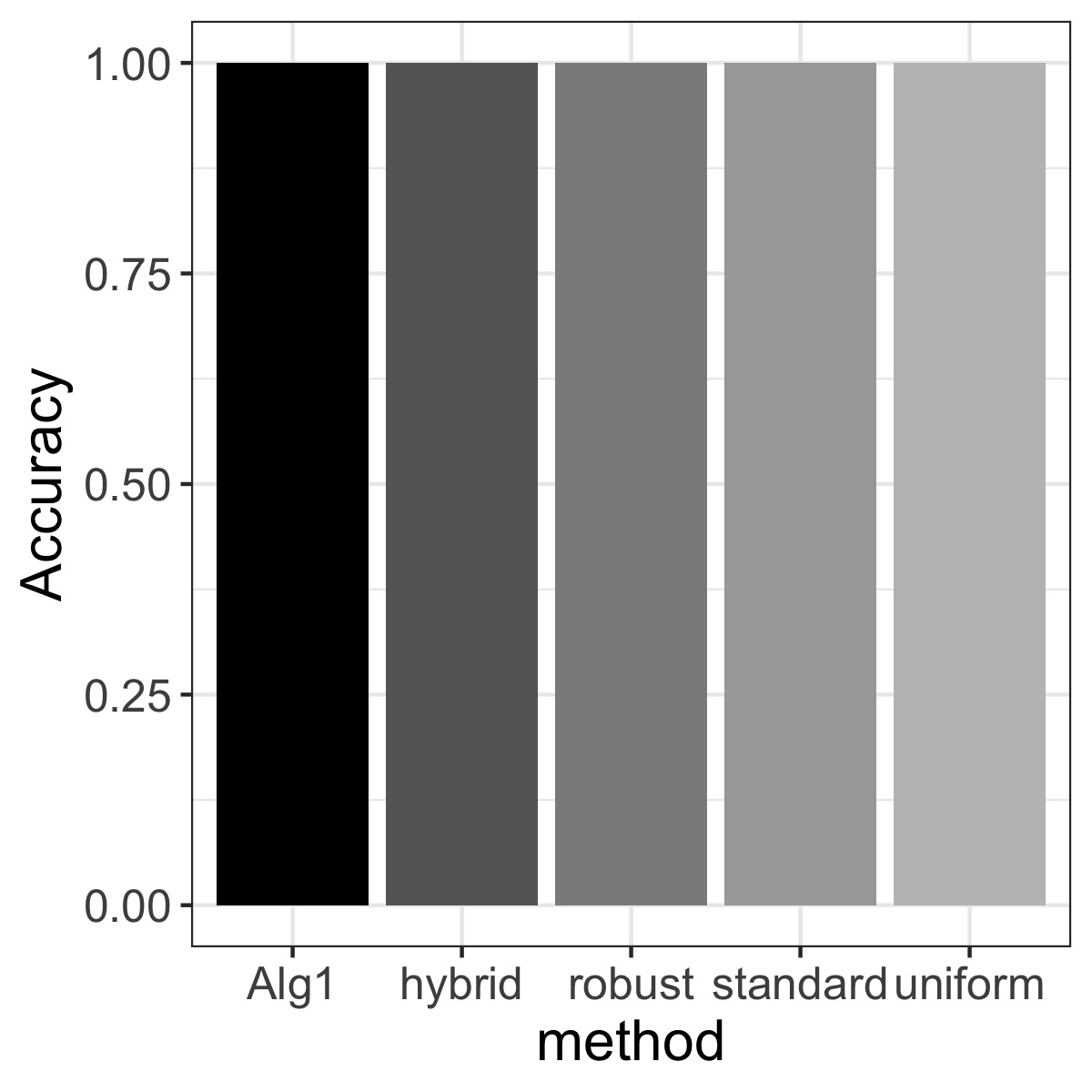}
		\end{subfigure}
		\begin{subfigure}[b]{0.3\columnwidth}
			\centering
      \caption{ $M_2$}
			\includegraphics[width=1.0\columnwidth]{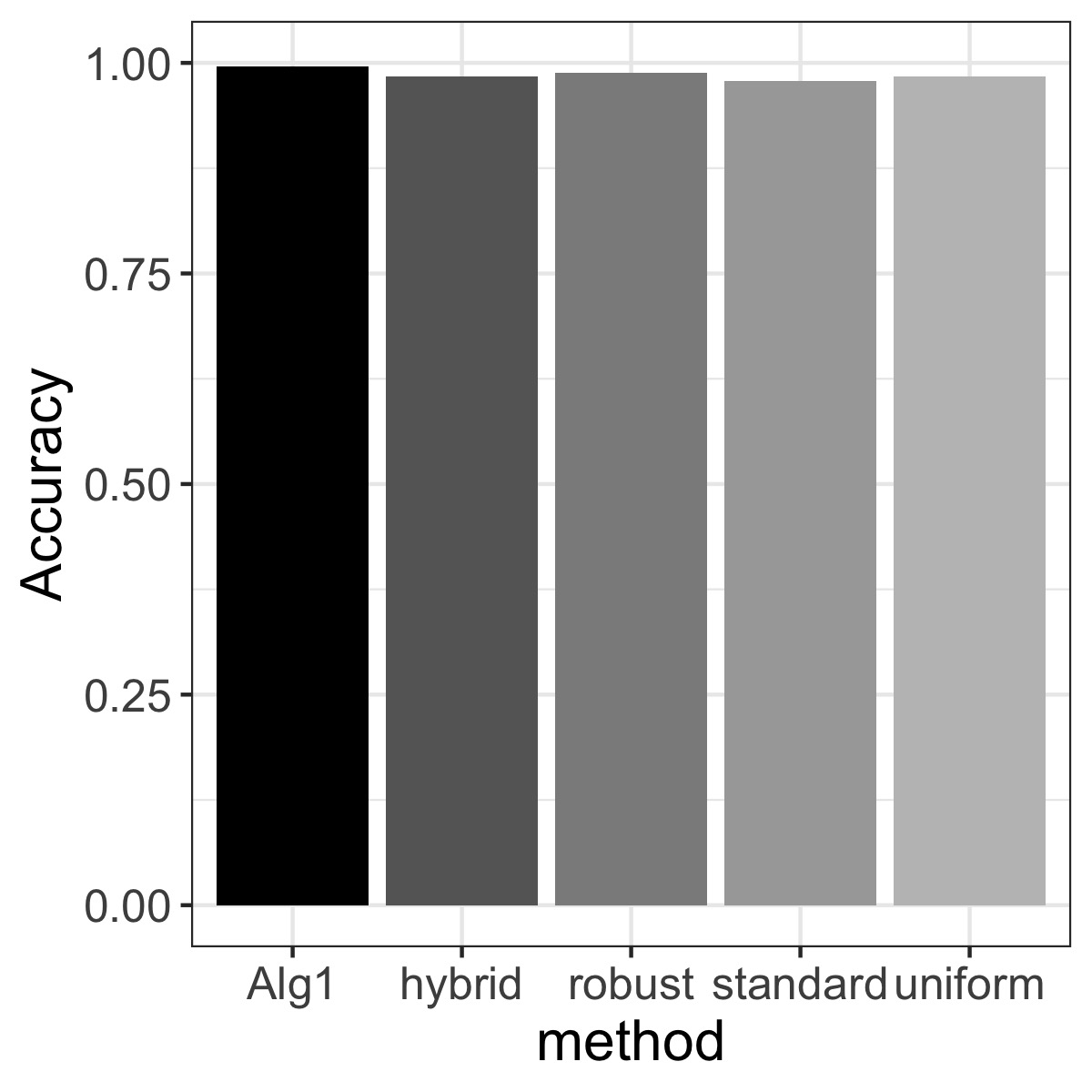}
		\end{subfigure}
		\begin{subfigure}[b]{0.3\columnwidth}
			\centering
      \caption{ $M_3$}
			\includegraphics[width=1.0\columnwidth]{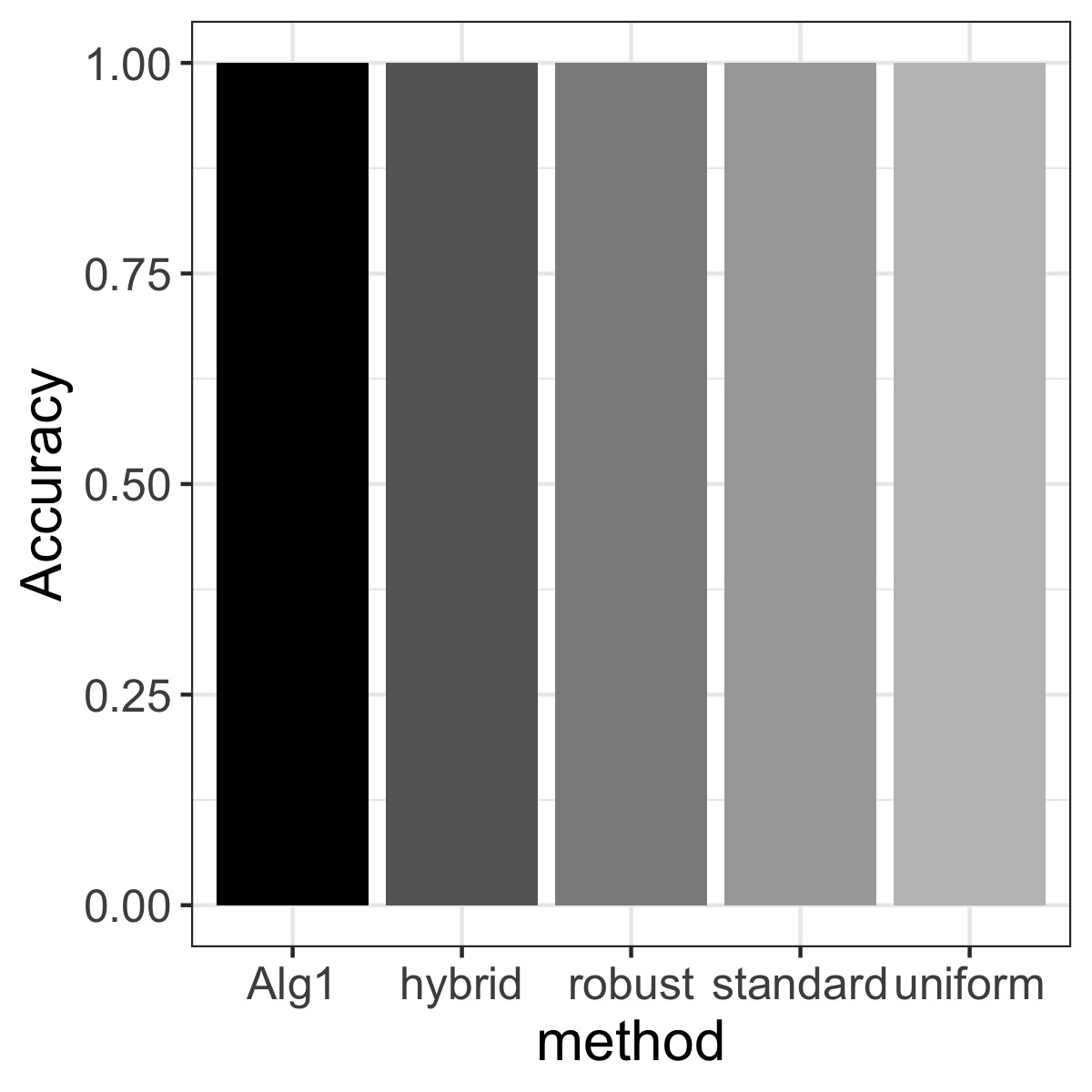}
		\end{subfigure}\\
  \begin{subfigure}[b]{0.3\columnwidth}
			\centering
         \caption{ $M_1$}
			\includegraphics[width=1.0\columnwidth]{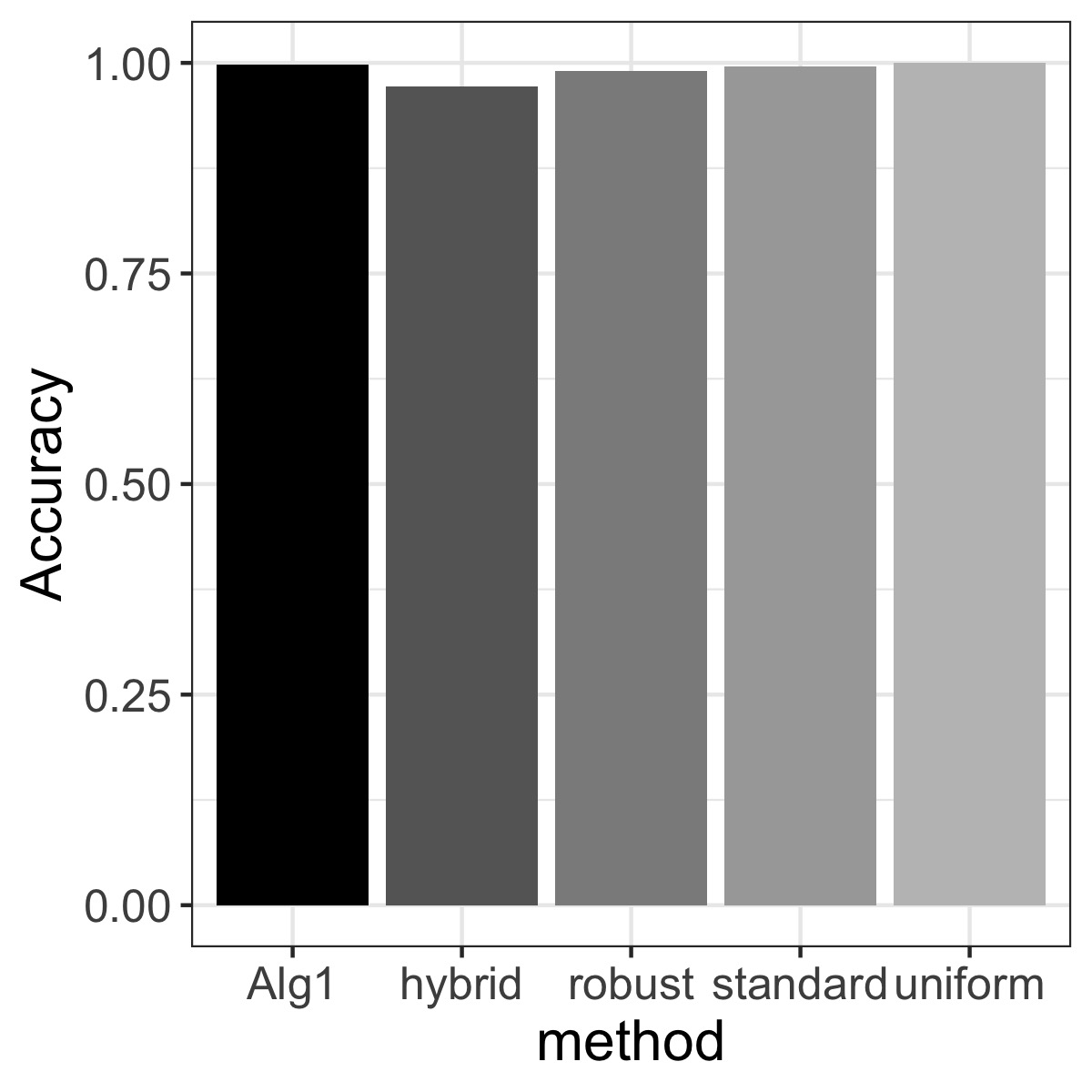}
		\end{subfigure}
		\begin{subfigure}[b]{0.3\columnwidth}
			\centering
         \caption{ $M_2$}
			\includegraphics[width=1.0\columnwidth]{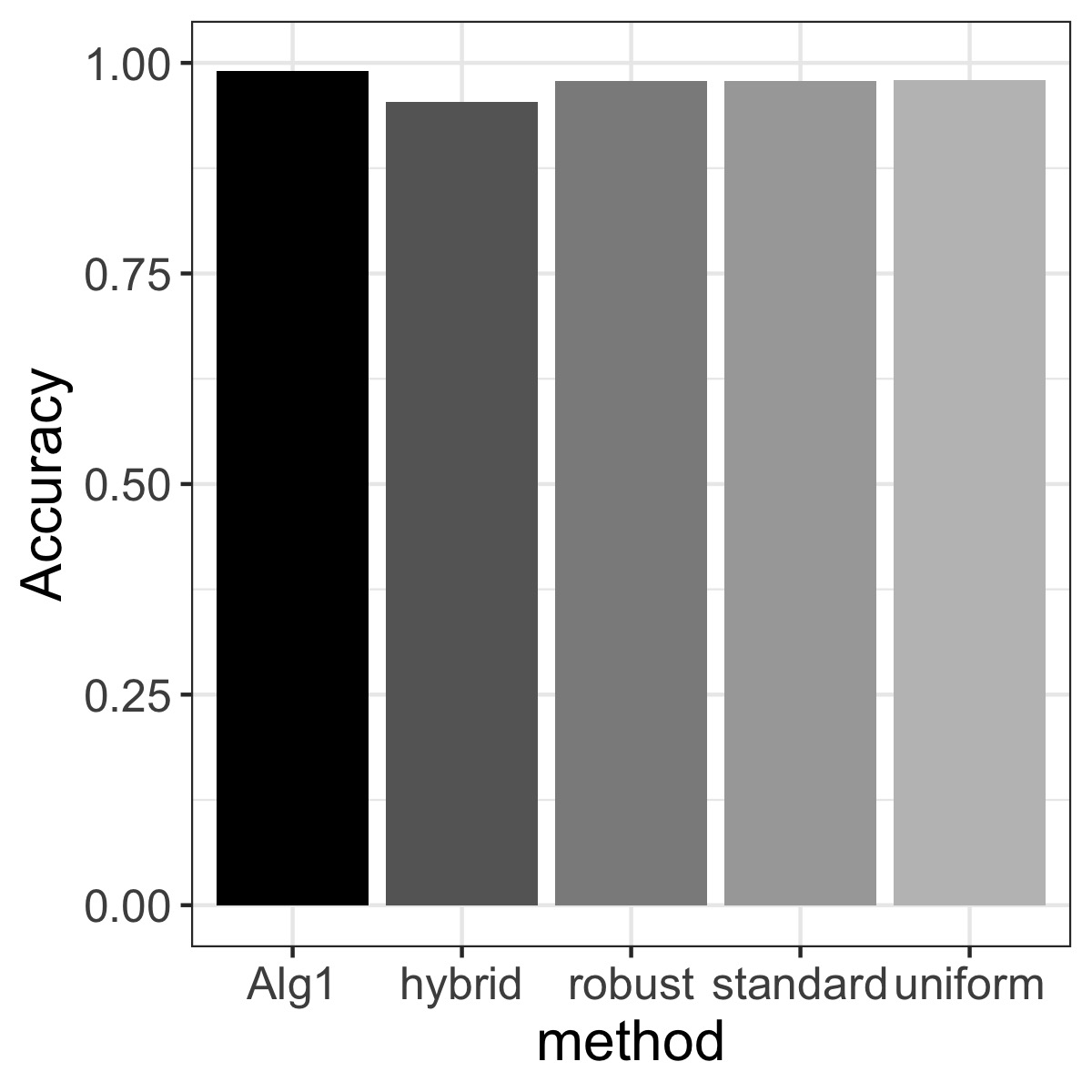}
		\end{subfigure}
		\begin{subfigure}[b]{0.3\columnwidth}
			\centering
         \caption{ $M_3$}
			\includegraphics[width=1.0\columnwidth]{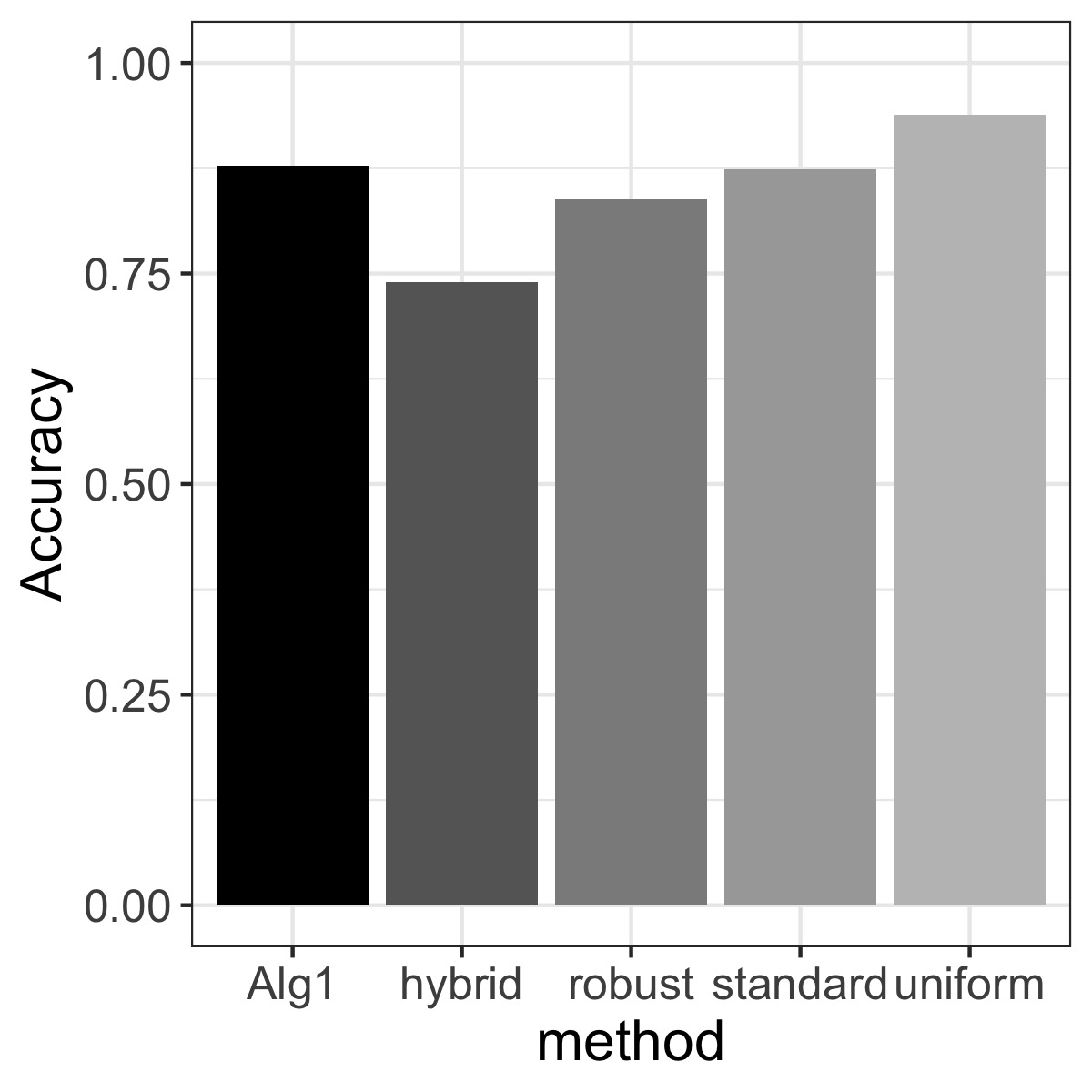}
		\end{subfigure}
		\caption{\label{fig:case3-barplot-acc}{
 \it 
 Selection accuracy of the BIC  based on $150$ observations from  different designs in  different true models defined  by \eqref{mod1}.   Upper part: SNR=3.75;  lower part: SNR=1.35.
 }}
	\end{figure}

\subsection{Models for robust parameter design}\label{robust}

    In a robust parameter design, the input variable are divided into two classes: control factors $(\bm x)$ and noise factors $(\bm z)$. Control factors refer to the design factors  whose values are fixed once they are chosen.  Noise factors refers to the design variables that are hard or expansive to control during the user conditions but can be observed.
    More details can be found in Chapter 11 of \cite{wu2021introduction}.
    Engineers aim to find a setting of control factors to make it less sensitive to noise variation. We consider a typical example 
    with three control factors and three noise factors and the model 
    \begin{align*}
     M_{\rm full}: Y=&\beta_0+\beta_1x_1+\beta_2x_2 + \beta_3x_3+(\beta_4+\beta_5x_1+\beta_6x_2+\beta_7x_3)z_1\\ &+(\beta_8+\beta_9x_1+\beta_{10}x_2+\beta_{11}x_3)z_2
      +(\beta_{12}+\beta_{13}x_1+\beta_{14}x_2+\beta_{15}x_3)z_3
        +\varepsilon,  
    \end{align*}
   to explore the response surface and  
   choose the control factor  $\bm x$ such that the coefficients of $\bm z$, for example $\beta_4+\beta_5x_1+\beta_6x_2+\beta_7x_3$, $\beta_8+\beta_9x_1+\beta_{10}x_2+\beta_{11}x_3$, $\beta_{12}+\beta_{13}x_1+\beta_{14}x_2+\beta_{15}x_3$, are close to zero.
   In practice, one may not known which noise factor has impact to the response.
   Thus, we consider the following seven reduced model as candidates
   \begin{align} \nonumber
       M_1:&~ Y=\beta_0+\beta_1x_1+\beta_2x_2 + \beta_3x_3+\varepsilon,  \\
   \nonumber    M_2:&~ Y=\beta_0+\beta_1x_1+\beta_2x_2 + \beta_3x_3+(\beta_4+\beta_5x_1+\beta_6x_2+\beta_7x_3)z_1 
        +\varepsilon,\\
  \nonumber     M_3:&~ Y=\beta_0+\beta_1x_1+\beta_2x_2 + \beta_3x_3+(\beta_4+\beta_5x_1+\beta_6x_2+\beta_7x_3)z_2 
        +\varepsilon,\\
  \label{mod2} 
  M_4:&~ Y=\beta_0+\beta_1x_1+\beta_2x_2 + \beta_3x_3+(\beta_4+\beta_5x_1+\beta_6x_2+\beta_7x_3)z_3 
        +\varepsilon,\\
\nonumber     M_5 :&~ Y=\beta_{0}+\beta_{1}x_1+\beta_{2}x_2+\beta_{3}x_3+(\beta_{4}+\beta_5x_1+\beta_6x_2)z_1+(\beta_{7}+\beta_{8}x_1+\beta_{9}x_3)z_3+\varepsilon,\\
\nonumber   M_6 :&~ Y=\beta_{0}+\beta_{1}x_1+\beta_{2}x_2+\beta_{3}x_3
   +(\beta_{4}+\beta_5x_1+\beta_6x_2)z_1+(\beta_{7}+\beta_{8}x_2+\beta_{9}x_3)z_3
   +\varepsilon,\\
\nonumber     M_7 :&~ Y=\beta_{0}+\beta_{1}x_1+\beta_{2}x_2+\beta_{3}x_3+(\beta_{4}+\beta_5x_1+\beta_6x_3)z_1+(\beta_{7}+\beta_{8}x_2+\beta_{9}x_3)z_2+\varepsilon,
   \end{align}
   where both the control  and noise factor 
   vary in the interval $[-1,1]$.
    In this case, the $D$-optimal design for the model $M_{\rm full}$ is a product designs of  two $2^3$ full factorial designs with eight replications, and this design  is also  $D$-optimal for the   models $M_1, \ldots ,  M_7$. However, when model $M_1$ is the true model, one can just use a $2^3$ full factorial design with only $64$ replications to achieve the same statistical efficiency.
   Clearly, the later one is more appreciated since it certainly reduces the costs in controlling the noise factors
  
  We now illustrate the application of  Algorithm~\ref{alg:vanillathompson1}, where  the total number of experiment units is fixed as $n=512$ and the  size  $n_t$ of observations at each stage $t$ depends on  the design we adopt. To be precise, when we choose an  optimal design for the models $M_1, \ldots M_4$,  
  at stage $t$, we set $n_t=16$. If we choose 
   the optimal design for the models  $M_5, \ldots M_7$ we set $n_t=32$. When a component $z_j$  of $\bm z$ is not controlled by the selected design at stage $t$, we assume that it is generated from a uniform distribution on the interval $[\min(x_j,0),\max(x_j,0)]$.
   Since the noise factor can be easily observed, we also assume that the value of $\bm z$ 
   will be recorded such  that all candidate models are estimable. Therefore  we  use  Algorithm \ref{alg:vanillathompson1} with 
 $\rho_t =0$   for all $t$.

The $D$-efficiencies  of the sequential  design generated by Algorithm \ref{alg:vanillathompson1}  are shown  in Figure~\ref{fig:case_robust}  for  an increasing
total sample size $n$, where  the true model is given by  $M_1$, $M_2$, and $M_5$. 
 As in Section \ref{ex:3}, we compare the new method with  the   
 robust optimal design which maximize the geometric mean of the $D$-efficiencies 
 in the  models  $M_1 - M_7$, 
the  uniform design with $64$ support points, and  a hybrid design defined by  $\sum_{k=1}^77^{-1}\xi_{k}^*$. 

The uniform design has rather low efficiencies. If model $M_1$ is the true model, all  other designs have efficiency one. If $M_2$ or $M_5$ is the true model, the efficiency of the sequential design increases quickly and exceeds the efficiency of the hybrid design already for sample sizes {$48$ and $80$} (usually after three or four stages), respectively.
Moreover, if the sample size is further increased the efficiency approaches $1$. For example, of $n=256$, the efficiencies of the sequential design are approximately $ 97\%$, if the true model is given by $M_2$ or $M_5$.

To further compare the optimal design $\xi^*_{\true}$  for the true model and the robust optimal design   with the sequential designs  generated by  Algorithm~\ref{alg:vanillathompson1} we  evaluate 
their costs in Figure~\ref{fig:case_robust-barplot-cost}.  
To be precise, the costs are  calculated by $1+5m_z$ for each experimental unit where $m_z$ is the number of noise factors that  are controlled in the corresponding unit. We observe  that the new method  method has competitive performance with the optimal design $\xi^*_{\true}$ for the true model
(the $D$-efficiencies are larger than $98\%$). 
Moreover, the sequential design generated by Algorithm \ref{alg:vanillathompson1}  and  the design $\xi^*_{\true}$ (which can only be used if the true model is known before the experiment) ~require a substantially smaller  budget compared to  the robust optimal design, because this design controls some redundant noise factors.

   \begin{figure}[t!]
\begin{subfigure}[b]{0.3\columnwidth}
			\centering
       \caption{$M_1$}
			\includegraphics[width=1.0\columnwidth]{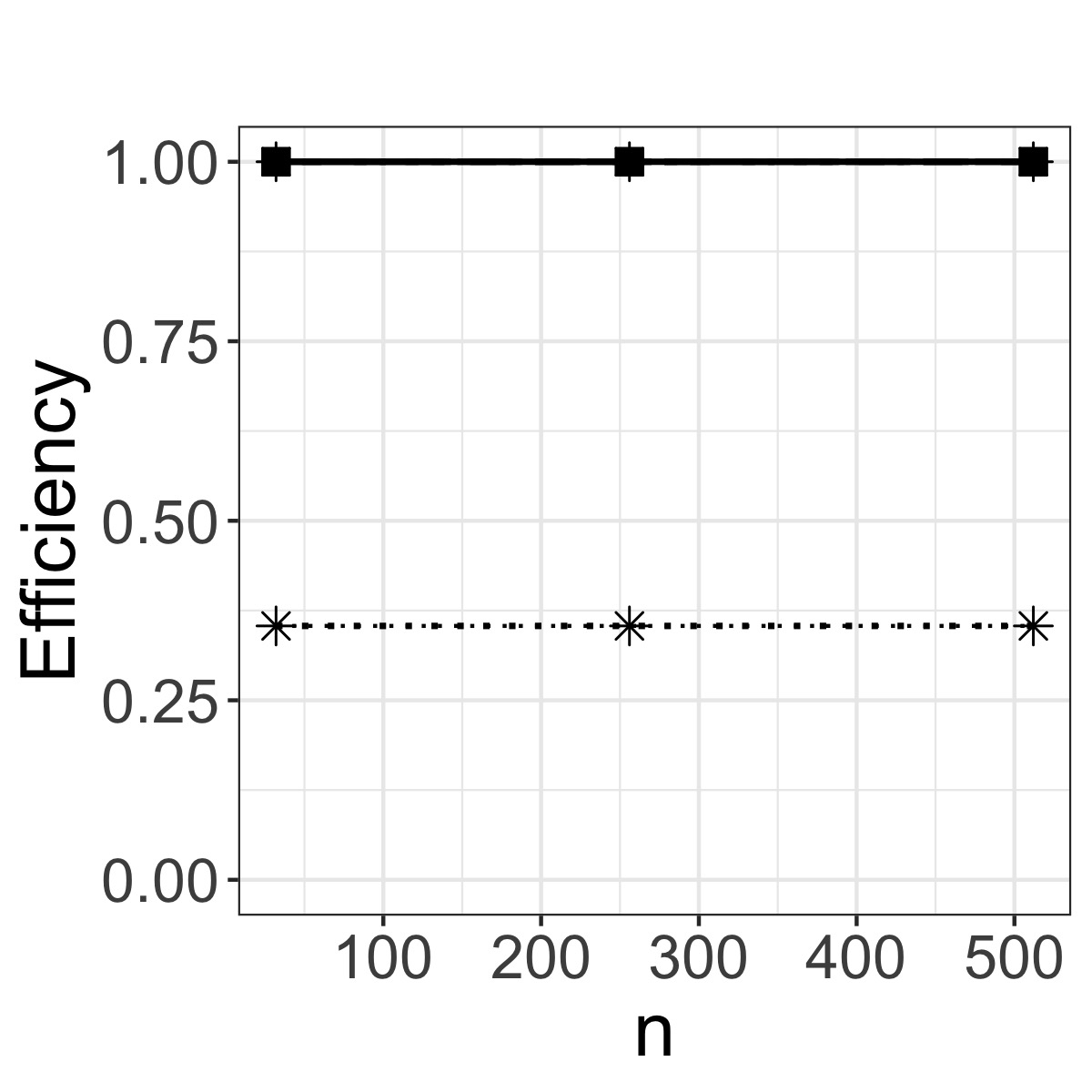}
		\end{subfigure}
		\begin{subfigure}[b]{0.3\columnwidth}
			\centering
       \caption{$M_2$}
			\includegraphics[width=1.0\columnwidth]{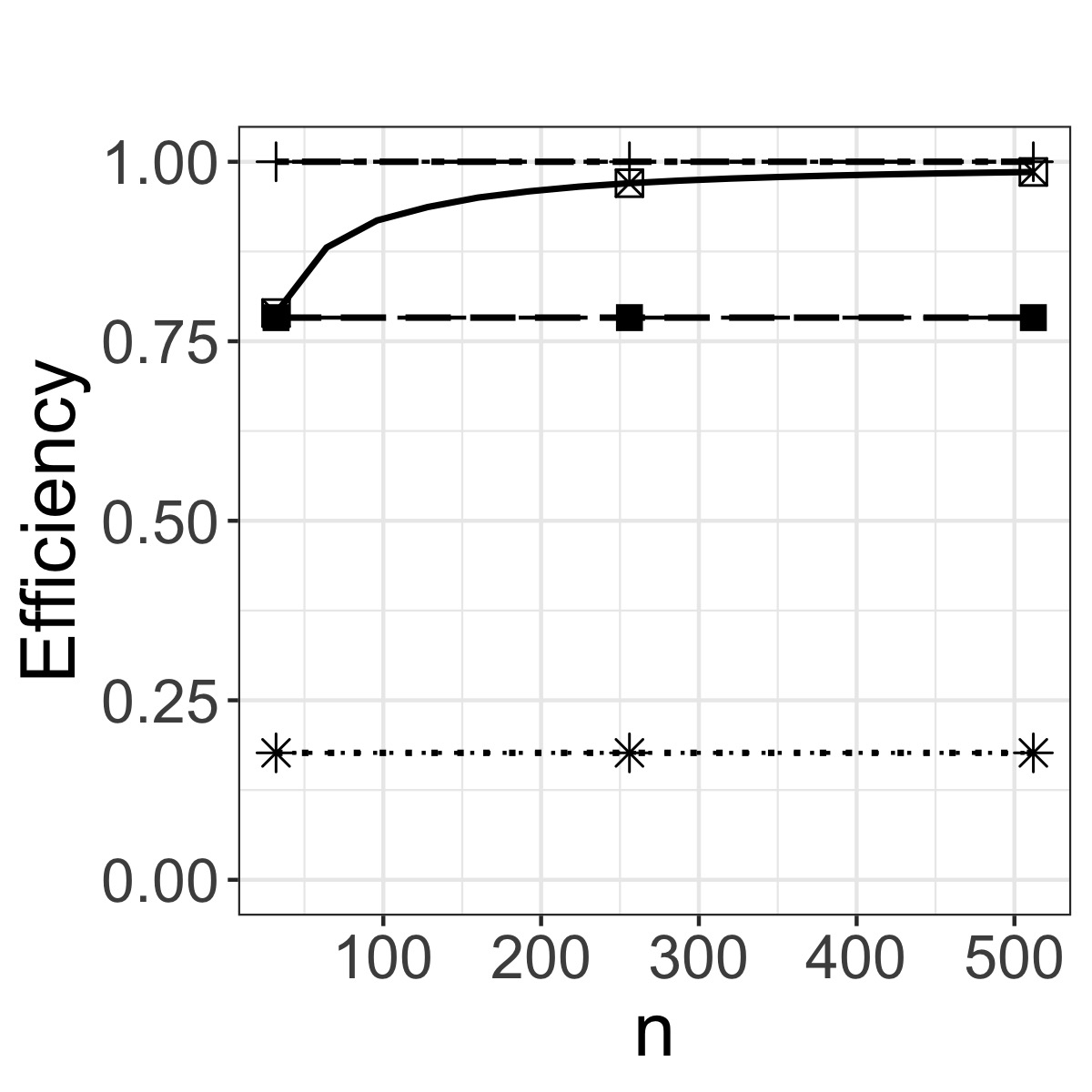}
		\end{subfigure}
		\begin{subfigure}[b]{0.3\columnwidth}
			\centering
    \caption{$M_5$}
			\includegraphics[width=1.0\columnwidth]{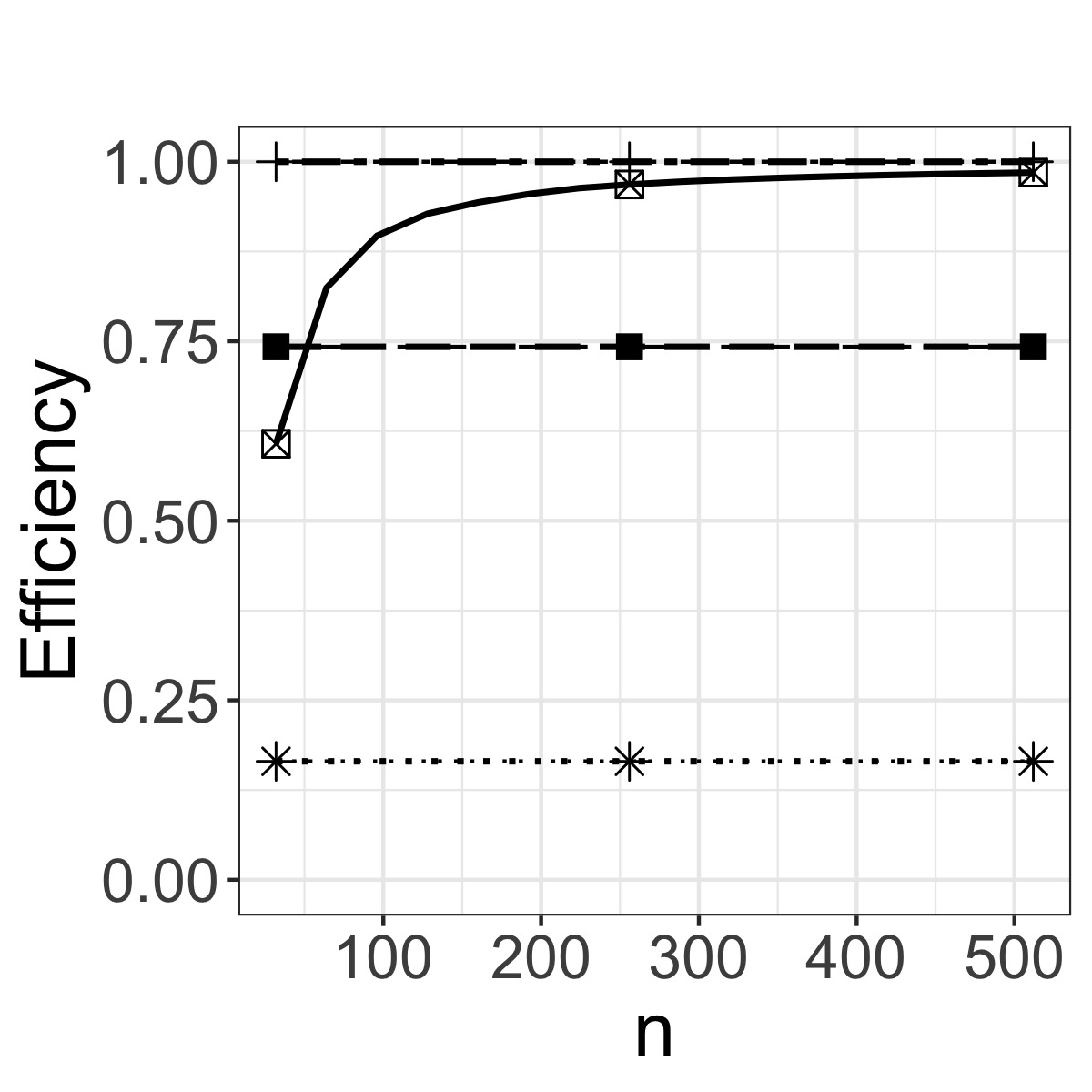}
		\end{subfigure}
\caption{\label{fig:case_robust}{
  \it $D$-efficiencies \eqref{det12} for the true model 
  as the the total sample size $n$  of the sequential design is increasing.  The models are defined by \eqref{mod2}.
  Four designs are compared: sequential design constructed by Algorithm \ref{alg:vanillathompson1} ($\boxtimes$);  robust design ($+$);  uniform design with 64 supports  ($\ast$);   hybrid design ($\blacksquare$).   
  }}

 \end{figure}

   \begin{figure}[t!]
\begin{subfigure}[b]{0.3\columnwidth}
			\centering
       \caption{$M_1$}
			\includegraphics[width=1.0\columnwidth]{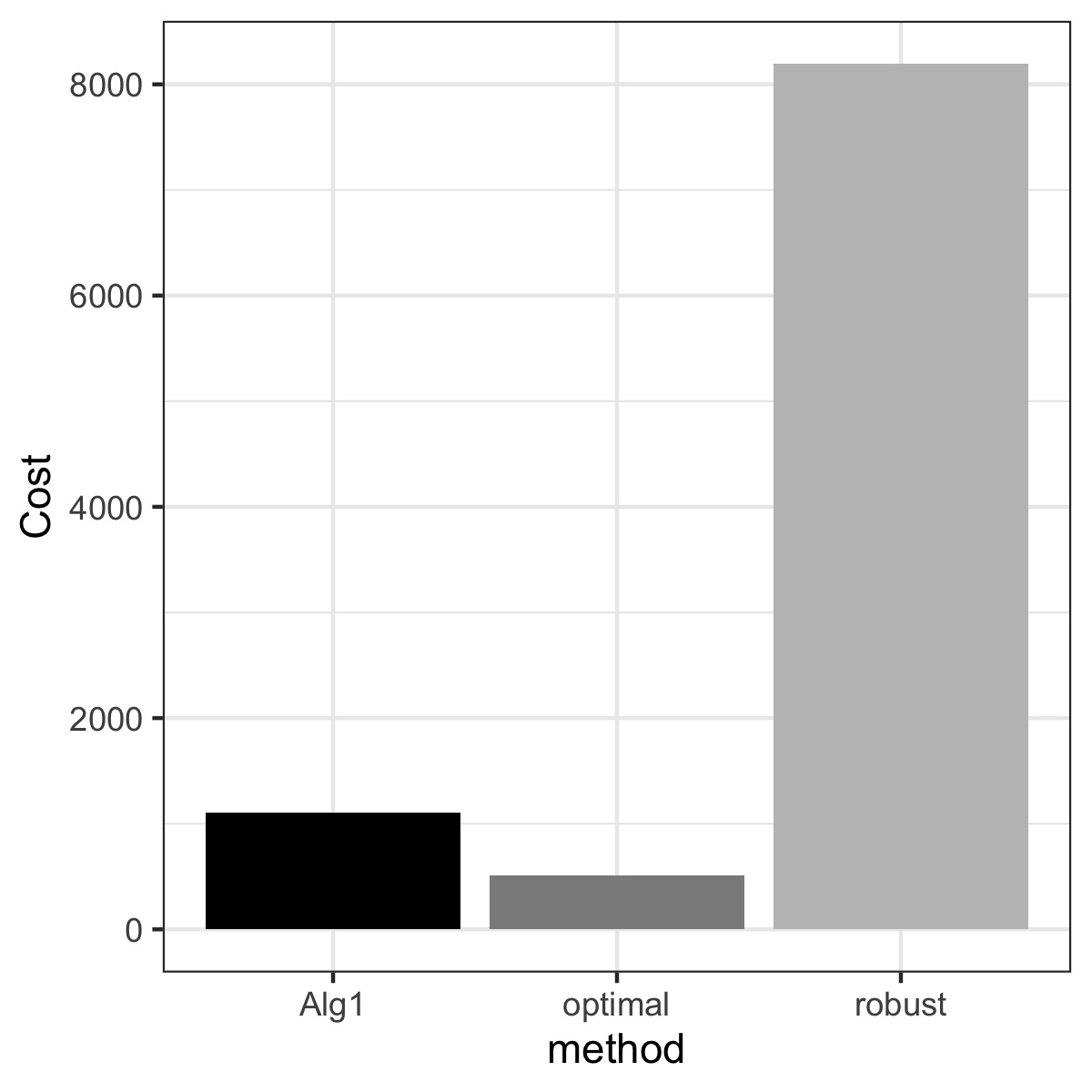}
		\end{subfigure}
		\begin{subfigure}[b]{0.3\columnwidth}
			\centering
       \caption{$M_2$}
			\includegraphics[width=1.0\columnwidth]{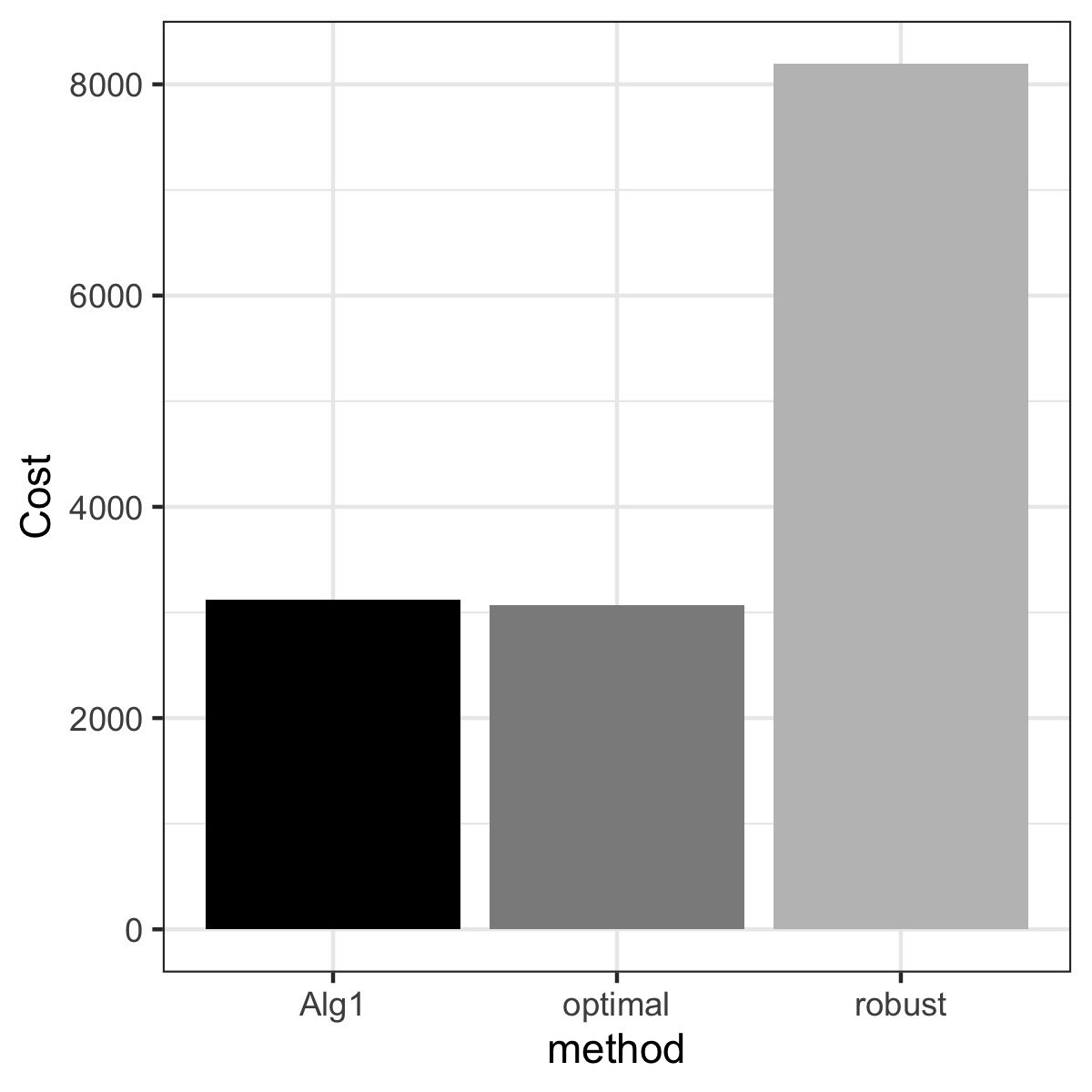}
		\end{subfigure}
		\begin{subfigure}[b]{0.3\columnwidth}
			\centering
       \caption{$M_5$}
			\includegraphics[width=1.0\columnwidth]{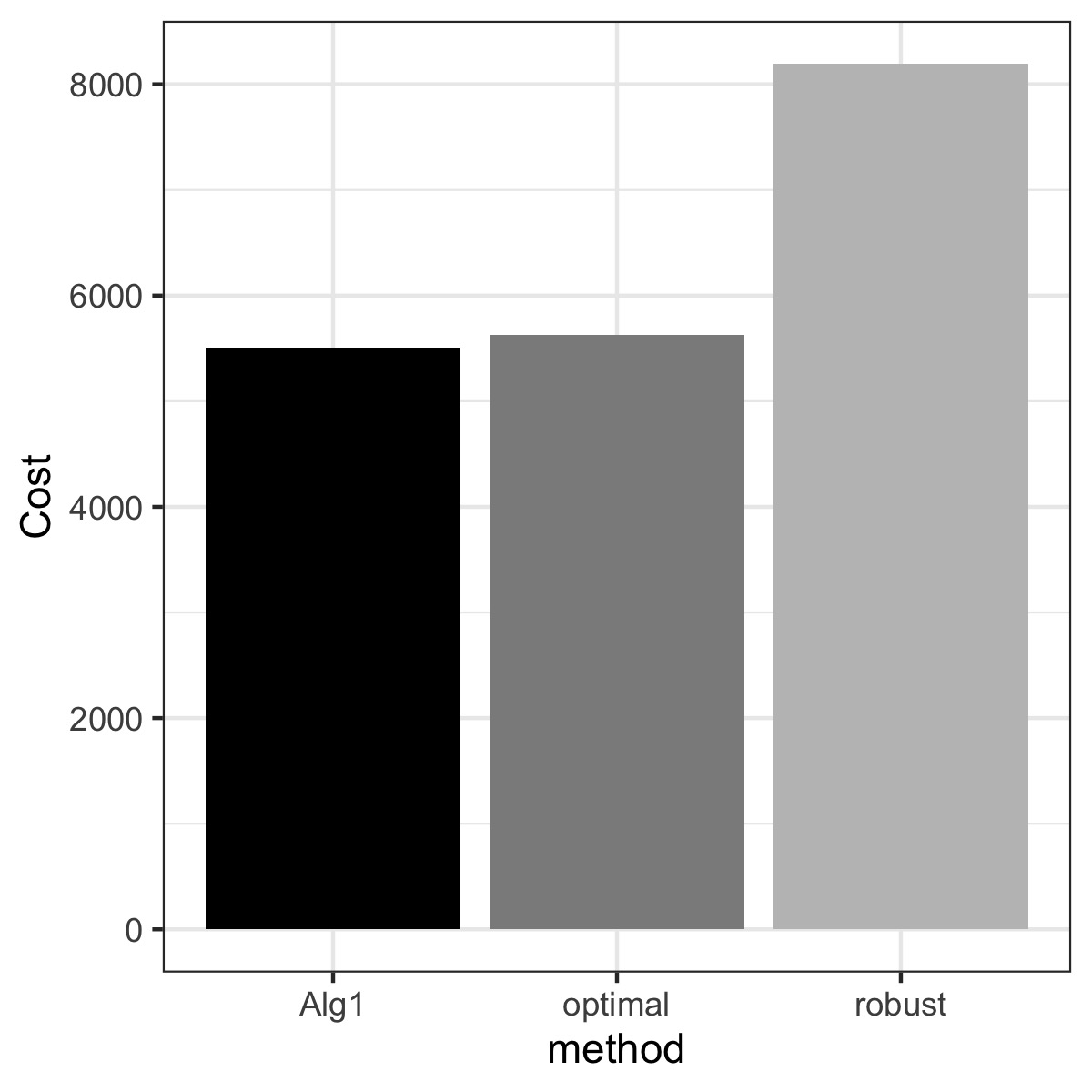}
		\end{subfigure}
\caption{\label{fig:case_robust-barplot-cost}{
  \it Design costs  for different methods   under $512$ experiment units when $M_1,M_2$, and $M_5$ is the true model, respectively. The models are defined by \eqref{mod2}.
  Here the results for the uniform design are omitted since the design efficiencies are not comparable with other methods.
  }}

 \end{figure}

\subsection{Models with multivariate predictors}\label{multivar}

In this subsection, we will illustrate how the modified Algorithm~\ref{alg:vanillathompson1} according to Remark~\ref{remark:testset1} works.
Consider the following six linear regression with the design region ${\cal X} =[-1,1]^3$.
\begin{align}
\nonumber
		M_1 :&~~ Y=\beta_{0} +\beta_{1}x_1+ \beta_{2}x_2+\beta_{3}x_3+\varepsilon,\\
	\nonumber	M_2 :&~~ Y=\beta_{0} +\beta_{1}x_1+ \beta_{2}x_2+\beta_{3}x_3+\beta_{4}x_1x_2+\beta_{5}x_3^2+\varepsilon,\\
	\label{mod3} 
  M_3 :&~~ Y=\beta_{0}+\beta_{1}x_1+\beta_{2}x_2+\beta_{3}x_3+\beta_{4}x_2x_3+\beta_{5}x_1^2+\varepsilon,\\
 \nonumber 
 M_4 :&~~ Y=\beta_{0}+\beta_{1}x_1+\beta_{2}x_2+\beta_{3}x_3+\beta_{4}x_1x_2+\beta_{5}x_1x_3+\varepsilon,\\
 \nonumber
 M_5 :&~~ Y=\beta_{0}+\beta_{1}x_1+\beta_{2}x_2+\beta_{3}x_3+\beta_{4}x_1x_2+\beta_{5}x_2x_3+\varepsilon,\\
 \nonumber
    M_6 :&~~ Y=\beta_{0} + \beta_{1}x_1+\beta_{2}x_2 + \beta_{3}x_3+\beta_{4}x_1x_2+\beta_{5}x_2x_3+\beta_{6}x_1x_3 +\beta_{7}x_1^2+\beta_{8}x_2^2+\beta_{9}x_3^2 +\varepsilon.
	\end{align}
According to Chapter  {15.11} of \cite{pukelsheim2006optimal}, a design with  equal masses at the  points $(1,1,1),(1,1,-1),(1,-1,1),(1,-1,-1),(-1,1,1),(-1,1,-1),(-1,-1,1)$, and $(-1,-1,-1)$ is {D-optimal}  for model $M_1, M_4$, and $M_5$. 
However, with this design  the models  $M_2$ and $M_6$ are not estimable 
since the  quadratic term is the alias of the intercept term.
Note that this problem can also be easily solved by Algorithm~\ref{alg:vanillathompson1} with a hybrid design but we use the modified algorithm  here  for illustrative  purposes.

Recall that  
 ${\cal D}_t=\{(\bm x_{n_0 + \ldots +  n_{t-1}+1}, Y_{n_0 + \ldots +  n_{t-1}+1}),\ldots, (\bm x_{n_0  + \ldots + n_{t}},Y_{n_0  + \ldots + n_{t}})\}$ is the set of $n_t$ observations  taken at stage $t$. We denote by $\bm { \tilde x_{ti}} , \ldots  ,
 \bm { \tilde x_{tm}}$ the different experimental conditions in ${\cal D}_t$ and
 by  $Y_{tij}$ ($i=1, \ldots , n_{ti}$) the observations taken at each  $\bm { \tilde x_{ti}}$ ($i=1, \ldots , m$) such that $\sum_{i=1}^m n_{ti}=n_t$.
As suggested in Chapter 2.3 in \cite{mccullagh1989generalized}, we  use Pearson $\chi^2$ statistic
\begin{equation*}
    \chi^2_t=\sum_{i=1}^{m}\sum_{j=1}^{n_{ti}}(Y_{tij}-  \hat{\mu}_{i,M_k} )^2/\hat{\sigma}^2,
\end{equation*}
for testing the goodness-of-fit of the plausible model $M_k$. 
Here $\hat{\mu}_{i,M_k}$ 
 is the prediction  at the point $\bm { \tilde x_{ti}}$ using  model $M_k$  estimated 
from the data ${\cal D}_t$,
$\hat{\sigma}^2={(n_t-m)^{-1}}\sum_{i=1}^{m}\sum_{j=1}^{n_{ti}}(Y_{tij}-\bar{Y}_{ti})^2$
is an estimate of the variance 
and   $\bar{Y}_{ti}=n_i^{-1}\sum_{j=1}^{n_{ti}}Y_{tij}$.

As in Section~\ref{ex:3}, we also compare the proposed design with a robust design which maximizes the geometric mean of the $D$-efficiencies   in model $M_1$--$M_6$, a  hybrid design  defined by $6^{-1} \sum_{k=1}^6\xi^*_{k}$ and a uniform design with $8$ supports. In this example, we set $n_t=36$ and $T=15$. The results are shown in Figure \ref{fig:gof},
where  the errors are standard normal distributed   and all parameters $\beta_j$ are set to  one. We observe that for an increasing number of stages the proposed designs are close to the optimal design under the true model in all cases under consideration. This  confirms our theoretical findings in 
 in Theorem~\ref{thm:optimal}.
For models $M_1$ and $M_5$ the sequential design already outperforms 
 its competitors for small sample sizes. On the other hand, 
if  the true model is the widest model $M_6$ and $T$ is relatively small, the robust methods have larger $D$-efficiency  since at the beginning the proposed method takes some designs with zero efficiency in  true model.
This is the price we pay for model exploration.

  \begin{figure}[t!]
\begin{subfigure}[b]{0.3\columnwidth}
			\centering
   \caption{$M_1$ }
     \vspace{-.25cm}
			\includegraphics[width=1.0\columnwidth]{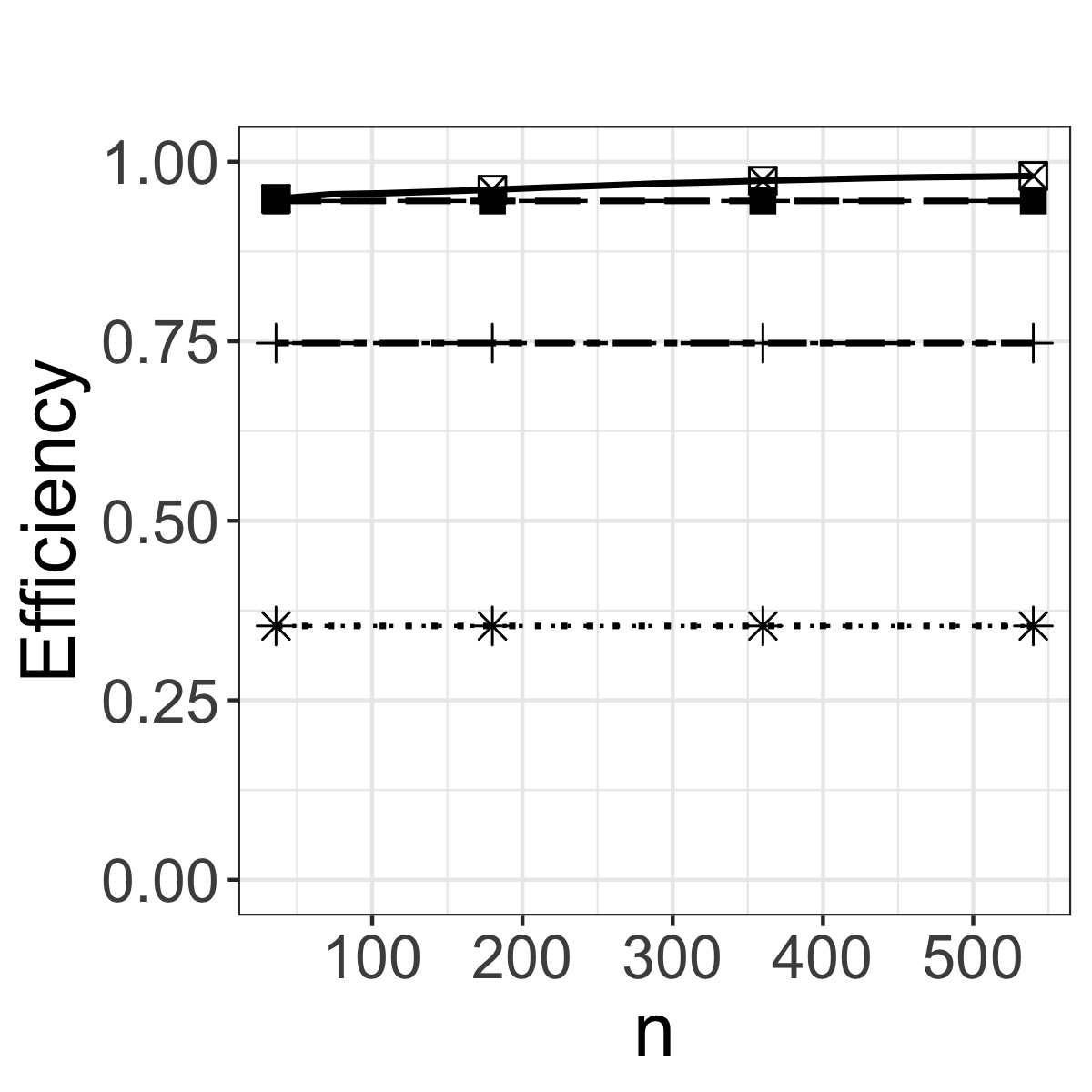}
		\end{subfigure}
		\begin{subfigure}[b]{0.3\columnwidth}
			\centering
\caption{$M_5$ }
    \vspace{-.25cm}
			\includegraphics[width=1.0\columnwidth]{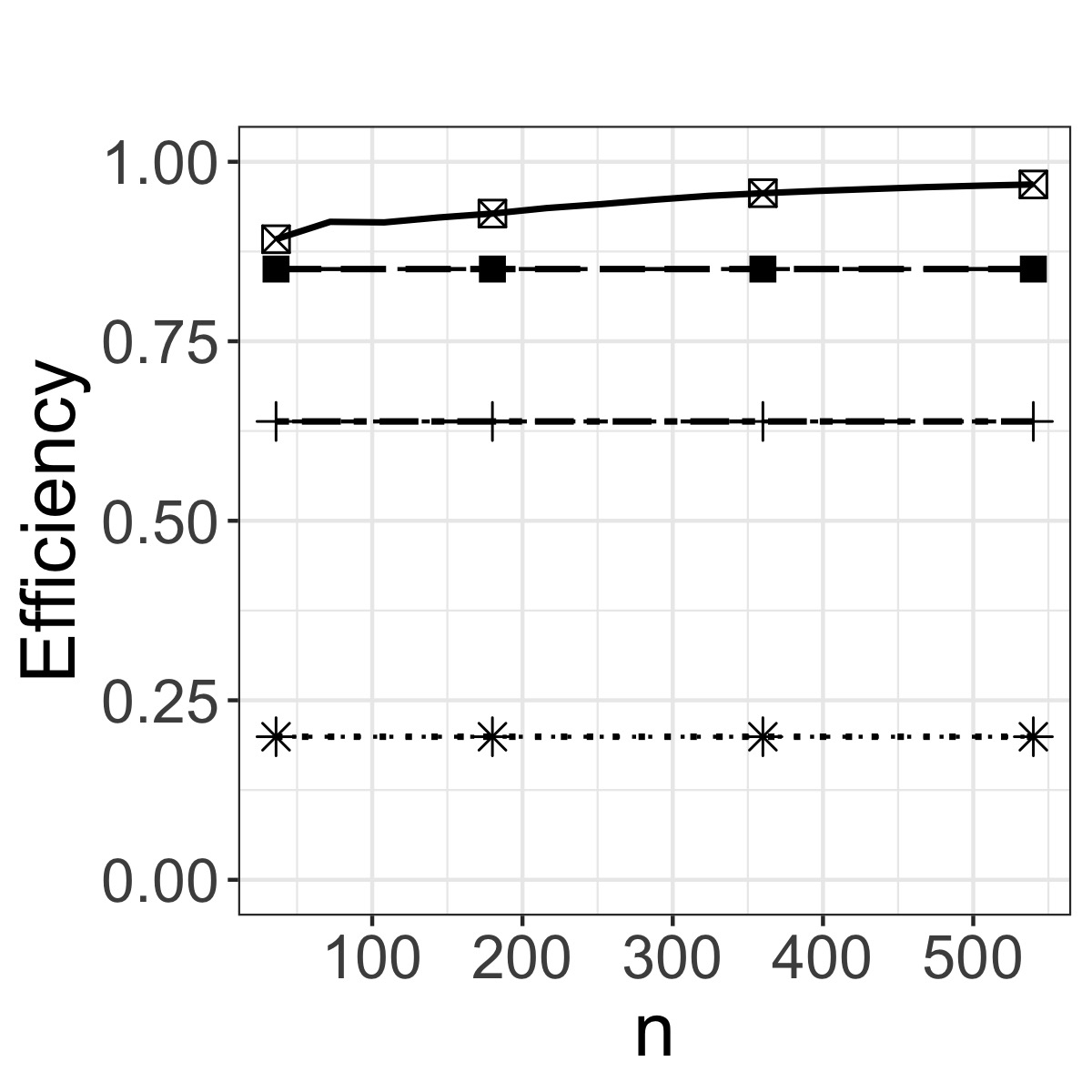}
		\end{subfigure}
		\begin{subfigure}[b]{0.3\columnwidth}
			\centering
   \caption{$M_6$  }
       \vspace{-.25cm}
			\includegraphics[width=1.0\columnwidth]{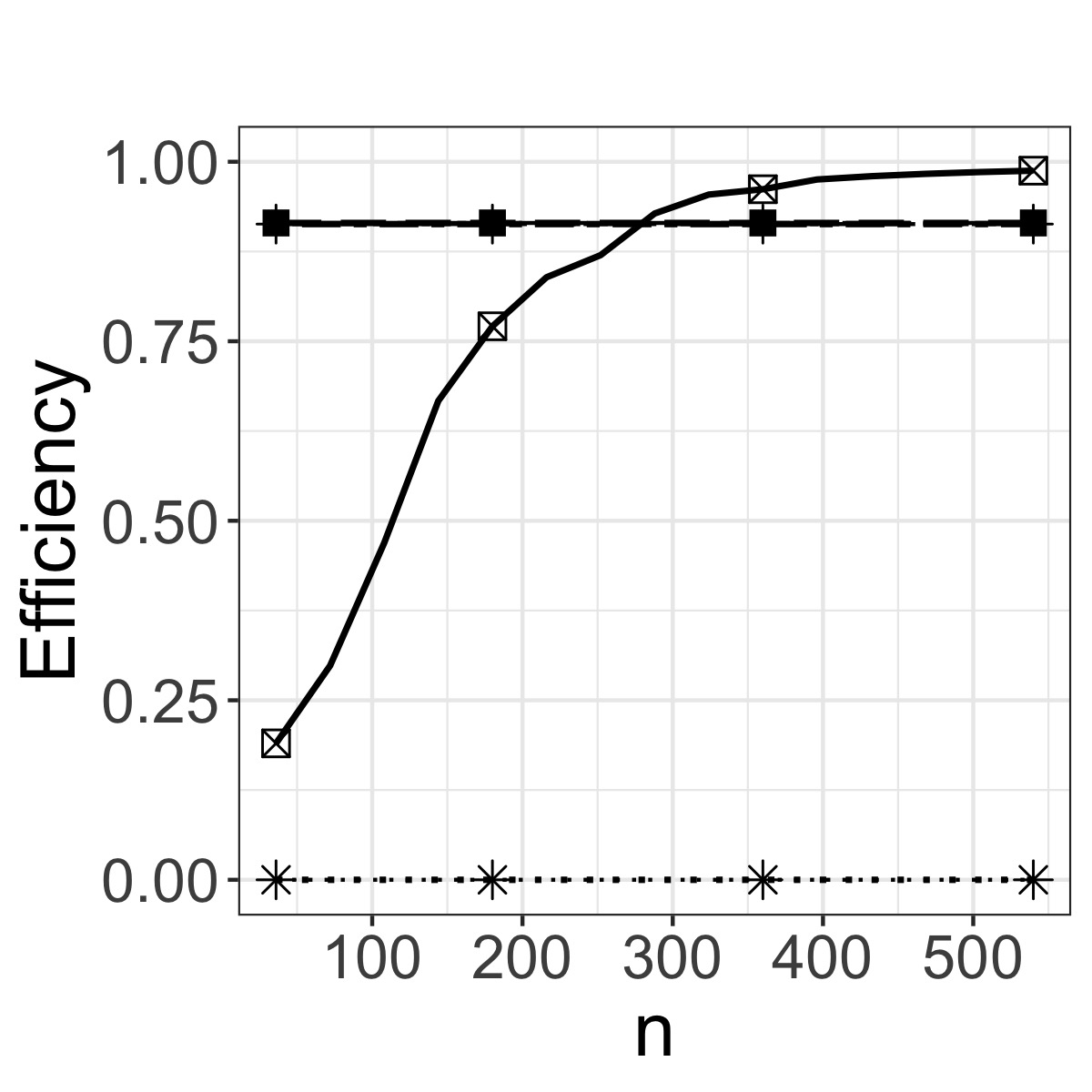}
		\end{subfigure}
\caption{\label{fig:gof}{
  \it $D$-efficiencies \eqref{det12} for the true model 
  as the total sample size $n$  of the sequential design is increasing. The models are defined by \eqref{mod3}.
 Four designs are compared: sequential design constructed by the modified Algorithm \ref{alg:vanillathompson1} according to Remark~\ref{remark:testset1} ($\boxtimes$);  uniform design with five supports  ($\ast$);  robust design($+$); and hybrid design ($\blacksquare$). 
  }}
  \end{figure}

\section{Concluding remarks}\label{sec:conclusion}

Optimal designs are frequently   criticized for their dependence on model specification.
In practice, it is common to see that several plausible models are available to 
describe the relation between the predictors  and  experimental outputs, and the goal is  
to  estimate the parameters in the 
 most appropriate (true) model among the candidate models.  
 Consequently, design of experiments  has to take the different goals,  most efficient 
 estimation in the true (but unknown) model and 
 identifying the true model, into account.
  In this paper  we address this problem 
by  a reinforcement learning  approach  
and construct  sequential designs which achieve a specified balance between estimation and discrimination.
The key idea is to feed only a small batch of experimental units at a time and
 to collect the gains to prioritize a design that is suitable for the true model, i.e., reinforcing the  positive outcomes.
From a theoretical  side  we prove that our method provides a sequence of design with efficiencies converging  $1$ for the true model. 
Moreover, we also derive finite sample lower bounds for the expected efficiency in true model, and show that the sequential algorithm also identifies the true model with a probability converging to one.
From a practical side we demonstrate by means of a simulation study that the new sequential designs have very good finite sample  properties compared to  other well-developed design procedures for model discrimination and parameter estimation.

In this paper we concentrate on linear models and 
locally optimal  designs for non-linear models, which require some prior information about the unknown parameters. An interesting direction of future research is the extension of the 
reinforcement learning approach
 to construct sequential designs with respect to
robust optimality criteria, such as  Bayes- 
maximin-criteria  \citep[see, for example][among many others]{chaver1995,dette1997,muepaz1998}.
Moreover,  it would be also interesting to study different strategies in  the evaluation step 
for power improvement. 

\bigskip

\textbf{ Acknowledgements.} 
The authors would like to thank Birgit Tormöhlen who typed parts of this manuscript with considerable technical expertise.
The work  of H. Dette  was partially supported by the    DFG Research unit 5381 {\it Mathematical Statistics in the Information Age}, project number 460867398.  

\bibliographystyle{apalike}
\bibliography{reference}

\appendix
\section{Appendix: Technical details}
\def\theequation{A.\arabic{equation}}
	\setcounter{equation}{0}

\subsection{ Proof of the results in Section \ref{sec41} }
\begin{proof}[Proof of Theorem~\ref{thm:optimal}]
This is a direct result from Theorem~\ref{thm:risk22} (which will be shown later) by letting $T\to\infty$.
\end{proof}

\begin{proof}[Proof of Theorem~\ref{thm:selection}]	
The first part comes from the result in Lemma~\ref{lem:times} (which will be shown later).
For the second part, note the fact that $\sum_{j=1}^T N_j(T)=T$. Thus it is sufficient to show that $N_{\true}(T)\ge T/2$ with probability approaching one.
	Without loss of generality, we assume $M_\true=M_1$.
	By the Markov's inequality,  
 it follows that
	\begin{align}
	\nonumber	\pr(N_1(t)<T/2)&=\pr\Big (T-\sum_{j=2}^\km N_j(t)<T/2\Big ) =\pr\Big  (\sum_{j=2}^\km N_j(t)>T/2\Big )\\
	\nonumber	&
 \le \frac{ 2}{T} \E \Big [ \sum_{t=1}^T \mathbb{I}(\mathcal{A}(t)) \Big ]  
  \le \frac{ 2}{T} (C_1\log(T)+C_2)
   \label{eq:14s} = o(1) ,
	\end{align}
 as $T \to \infty$, 
where  $C_1, C_2$ are some constants and the second inequality is a consequence of  Lemma~\ref{lem:times} (which will be shown later).
\end{proof}

\subsection{ Proof of the results in Section \ref{sec42} }
In this subsection, we will begin with dealing with the result in Theorems~\ref{thm:risk22}.
The proof of Theorem~\ref{thm:risk21} is postponed to the end of this section.

\subsubsection{Proof of Theorem 3}
Before proving  Theorem~\ref{thm:risk22}, we 
introduce some  addition notations and 
several auxiliary results. 
First, we define  $\hat{\theta}_{j}(t)=(a_j(t)-1)/(t-1)$ for $t>1$ as a ``rough estimator'' 
of the posterior probability 
$\theta_{j}$ in \eqref{det11} and define 
 $\hat{\theta}_{j}(1)=0$ for all $j$.
We further introduce the following three events
\begin{align}
 \mathcal{B}(t) &= \left\{\hat{\theta}_\true (t)\le  c +\frac{\alpha}{2} \text { or } \exists M_j \neq M_\true \text { s.t. } \hat{\theta}_j(t)\ge  c -\frac{\alpha}{2}\right\}, \\
 	\mathcal{C}(t) &= \left\{\hat{\theta}_\true (t)-\eta_{\true}(t)\ge \frac{\alpha}{2} \text { or } \exists M_j \neq M_\true \text { s.t. } \eta_j(t)-\hat{\theta}_j(t)\ge \frac{\alpha}{2}\right\}, \\
 \mathcal{D}(t) &=\Big \{t>\frac{8 \log(T)}{\alpha^2}+1\Big \}.
\end{align}

\begin{lemma}\label{lem:risk1}
 Under Condition~\ref{ass:stability}, it holds that
	\begin{equation} \label{det22}
		\risk_1(T) :=  \sum_{t=1}^T \E (\mathbb{I}(\mA(t)\cap\mB(t))) \le \km \Big (1+\frac{2}{\alpha^2}\Big ).
	\end{equation}
\end{lemma}

\begin{proof}[Proof of Lemma~\ref{lem:risk1}]
	 The  risk $\risk_1(T)$ can be decomposed as
 \begin{equation}\label{eq:20}
		\risk_1(T)\le\E\Big  [ \sum_{t=1}^T \mathbb{I}\left(\hat{\theta}_{\true}(t) \le {c} + \frac{\alpha}{2} \right)\Big  ]  +\sum_{M_j\neq M_\true}\E\Big [ \sum_{t=1}^T \mathbb{I}\left(\hat{\theta}_{j}(t) \ge {c} - \frac{\alpha}{2}\right) \Big ]  
	\end{equation}
 and for the sake of brevity, we only deal with the first term showing   that
	\begin{equation}\label{eq:term11}
		\E\Big  [ \sum_{t=1}^T \mathbb{I}\left(\hat{\theta}_{\true}(t) \le  c  + \frac{\alpha}{2} \right)\Big  ]  \leq 1 + \frac{2}{\alpha^2}.
	\end{equation}
	The inequality for the second  term in \eqref{eq:20} can be obtained analogously. 
	With the definition of $\theta_j(\xi_{t}^{\rm hyb})$ we denote by 
\begin{align*}
    \bar\theta_j(t)&=\frac{1}{t-1}\sum_{k=1}^{t-1} \theta_j(\xi_{(k)}^{\rm hyb})=\frac{1}{t-1}\sum_{k=1}^{t-1} \theta_j(\rho_k\xi_{\Set_k}^*+(1-\rho_k)\xi_{\rm unif})\\
    &= \frac{1}{t-1}\sum_{k=1}^{t-1}\E [ \zeta_j(t) ] =\frac{1}{t-1}\E [ a_j(t)-1 ] =\E [\hat{\theta}_{j}(t)] 
\end{align*}
the expectation of $\hat{\theta} (t)$
 for  {$t=2,\ldots,T$}.
Simple calculations and Condition~\ref{ass:stability} yield that 
(note that  $\hat{\theta}_\true (1) =0$ and  
 $\bar\theta_{\true}(t) > c+\alpha$  for all $t>1$) {
	\begin{align*}
		\E\Big [ \sum_{t=1}^T \mathbb{I}(
  \hat{\theta}_{\true}(t) \le c +\alpha/2)\Big ] 
		=& 1 + \sum_{t = 2}^T \pr\left(\hat{\theta}_{{\true}}(t) - 	\bar\theta_{{\true}}(t)\le c+\alpha/2- \bar\theta_\true (t) \right)\\
		\leq& 1 + \sum_{t = 2}^T \pr\left(\hat{\theta}_{{\true}}(t) - 	\bar\theta_{{\true}} (t)\le -\alpha/2\right) \\
  	\leq &  1+ \sum_{t = 1}^T \exp\left(-t\alpha^2/2\right)
		\leq 1+ \sum_{t = 1}^\infty \exp\left(-t\alpha^2/2\right)\\
		\leq  & 1 + \frac{\exp\left(- \alpha^2/2\right)}{1-\exp\left(- \alpha^2/2\right)}
		< 1+ \frac{2}{\alpha^2},
	\end{align*} 
 where we have used Hoeffding's  inequality  and  the last inequality comes from the fact that $1/(\exp(x)-1)<1/x$.}
	This yields \eqref{eq:term11}. The second term can be bounded analogously, which concludes the proof of this Lemma. 
\end{proof}

\begin{lemma} \label{lem:4} 
	 The following bound holds when the Algorithm~\ref{alg:vanillathompson1} is applied.
\begin{equation}
   \label{det25} 
	\risk_2(T):=\sum_{t=1}^T \E\{\mathbb{I}(\mathcal{A}(t)\cap \mathcal{B}^c(t)\cap\mathcal{C}(t)\cap \mathcal D(t))\}\leq  \km.
	\end{equation}
\end{lemma}
\begin{proof}[Proof of Lemma~\ref{lem:4}]
	For the indices in the set  $\mathcal D(t)$ we have $t-1> {8\log(T)}/{\alpha ^2}$, which implies
	${\alpha}/{2}> \big ({2\log(T)}/({t-1}) \big )^{1/2}$. 
 The event $\mathcal{B}^c(t)\cap \mathcal{C}(t)\cap \mathcal D(t)\subseteq \mathcal{C}(t)\cap \mathcal D(t)$ is 
 then a subset of  
	$$
	\Big \{\exists M_j \neq M_\true : \eta_j(t) - \hat{\theta}_j(t) \ge \Big ({\frac{2\log(T)}{{t-1}}}\Big  )^{1/2}~\text{or}~ ~~ \hat{\theta}_{\true}(t)- \eta_{\true}(t) \ge \Big ({\frac{2\log(T)}{{t-1}}}\Big )^{1/2} \Big \}. 
	$$  
	As  $\hat{\theta}_{j}(t)=(a_j(t)-1)/(t-1)=(a_j(t)-1)/(a_j(t)+b_j(t)-2)$  is the mode of a random variable sampled from   a $\mathrm{Beta}(a_j(t),b_j(t))$ distribution,  we can use 
 Lemma 4 in \cite{wang2018th}
 to obtain the bound
	\begin{align*}
		\sum_{t=1}^T \pr\left(\mathcal{A}(t) \cap \mathcal{B}^c(t)\cap \mathcal{C}(t)\cap \mathcal{D}(t)\right)&\leq \sum_{t=1}^T \pr\left( \mathcal{B}^c(t)\cap \mathcal{C}(t)\cap\mathcal{D}(t)\right) \leq 
  \km ,
	\end{align*}
which gives the desired result.
\end{proof}

Our next result provides a bound on the expected number of times that the event $\mA(t)=\{M_{\Set_t} \neq M_\true\}$ happens for $t=1,\ldots,T$.
\begin{lemma}\label{lem:times}
	 If Condition~\ref{ass:stability} is satisfied  
 and all models are estimable under the design selected by Algorithm~\ref{alg:vanillathompson1}, we have
	\begin{equation}\label{eq:16}
		\E \Big [ \sum_{t=1}^T \mathbb{I}\left(\mathcal{A}(t)\right) \Big ] 
  \le \km\Big (2+\frac{2}{\alpha^2}\Big )+\frac{8\log(T)}{\alpha^2}.
	\end{equation}
\end{lemma}
\begin{proof}[Proof of Lemma~\ref{lem:times}]
		We use the decomposition
	\begin{align}
\mathbb{I}\left(\mathcal{A}(t)\right) 
		& =   \mathbb{I}\left(\mathcal{A}(t) \cap\mathcal{B}(t)\right) +  \mathbb{I}\left(\mathcal{A}(t) \cap\mathcal{B}^c(t)\cap \mathcal{C}(t) \right) + \mathbb{I}\left(\mathcal{A}(t) \cap\mathcal{B}^c(t)\cap \mathcal{C}^c(t)\right). \label{eq:event-decomp}
	\end{align}
 and derive bounds  for the  three terms separately.  
 
 The first term has already been considered in Lemma~\ref{lem:risk1}. For the second
term  we use a  further decomposition
	\begin{align*}
		 \mathbb{I}\left(\mathcal{A}(t) \cap\mathcal{B}^c(t)\cap \mathcal{C}(t) \right) &=	
		  \mathbb{I}\left(\mathcal{A}(t) \cap\mathcal{B}^c(t)\cap \mathcal{C}(t) \cap  \mathcal{D}(t)  \right) 
		  + \mathbb{I}\left(\mathcal{A}(t) \cap\mathcal{B}^c(t)\cap \mathcal{C}(t)\cap  \mathcal{D}^c(t)  \right).
	\end{align*}
 To evaluate the second term of the above equality, note that the worst scenario
 appears in the case where  the designs selected in the first   
 $t \leq  8\log(T)/\alpha^2$ stages are all incorrect, which implies for the corresponding  risk 
 $$
 \E \Big [ \sum_{t=1}^T 
 \mathbb{I}\left(\mathcal{A}(t) \cap\mathcal{B}^c(t)\cap \mathcal{C}(t)\cap  \mathcal{D}^c(t)  \right)
 \Big ]  \leq 8\log(T)/\alpha^2. 
 $$
Now, an application of 
Lemma~\ref{lem:4} gives 
 \begin{align*}
  \E \Big [ \sum_{t=1}^T 
\mathbb{I}\left(\mathcal{A}(t) \cap\mathcal{B}^c(t)\cap \mathcal{C}(t) \right) 
 \Big ]  \leq \km+8\log(T)/\alpha^2 
 \end{align*}
Finally, we deal with the last term in \eqref{eq:event-decomp}.
Note that 
\[
   \mathcal{B}^c(t)\cap \mathcal{C}^c(t)
   {=} \big \{ {\eta_{\true} } (t) >c~\text{and}~\forall M_j \neq  M_\true \,\,\,\, {\eta_j(t)}  < c \big \},
\]
which implies that ${\eta_{\true} }(t)$ corresponds to the largest value among all $\{\eta_j(t)\}_{j=1}^\km$ at stage $t$. In other words Algorithm~\ref{alg:vanillathompson1}, selects the optimal design under the true model, which means that  $\mathcal{A}^c(t)$ happens at stage $t$.  
Thus, $\mathbb{I}(\mathcal{A}(t) \cap\mathcal{B}^c(t)\cap \mathcal{C}^c(t))=0$, and  the desired result follows by 
combining  the estimates for the  three terms in \eqref{eq:event-decomp}.
\end{proof}

\paragraph{Proof of Theorem \ref{thm:risk22}.}
Recall that the number of runs in each stage is  the same, which implies $n_t/n=1/T$ and that $\rho_t=N_{\Set_{t}}(t)/(N_{\Set_{t}}(t)+1)$,  where  ${\Set_{t}}$  denotes the index of the model selected at stage $t$
and  $N_j(t)$ is the number of times that design $\xi_j^*$ was  chosen  in  the first  $t$ stages.
Using  the concavity of the criterion function 
$\phi_{\true}(\cdot)$ yields  
 \begin{align*}
     \Eff_{\true}(\xi_{\ms,T})&=\frac{\phi_\true(\xi_{\ms,T})}{\phi_\true(\xi_{\true}^*)}
     =\frac{\phi_\true(\sum_{t=1}^T 
     {T}^{-1}
     \{\rho_t\xi^*_{\Set_t}+(1-\rho_t)\xi_{\rm unif}\})}{\phi_\true(\xi_{\true}^*)}\\
     &\ge \frac{\frac{1}{T}\sum_{t=1}^T  \rho_t\phi_\true(\xi^*_{\Set_t})}{\phi_\true(\xi_{\true}^*)}
     \ge
     \frac{1}{T}\sum_{t=1}^T  \rho_t\mathbb{I}\left(\mathcal{A}^c(t)\right)\\
     &=\frac{1}{T}\sum_{t=1}^T  \left(1-\frac{1}{N_{\true}(t)+1}\right)\mathbb{I}\left(\mathcal{A}^c(t)\right)
     \ge\frac{1}{T}\sum_{t=1}^T  \mathbb{I}\left(\mathcal{A}^c(t)\right)-\frac{1}{T} \sum_{t=1}^T  \frac{1}{t}\\
     & \ge\frac{1}{T}\sum_{t=1}^T  \left(1-\mathbb{I}\left(\mathcal{A}(t)\right)\right) {-} \frac{1}{T}\log(T),
 \end{align*}
where the second last third  inequality follows  from the fact that $1/2+1/3+\cdots+1/(N_{\true}(T)+1)\le 1+1/2+\cdots+1/T$, since $N_{\true}(T)\le T$.
	Thus the result follows immediately from  Lemma~\ref{lem:times}.

\subsubsection{Proof of Theorem \ref{thm:risk21}}

Note that in the situation considered in 
Theorem \ref{thm:risk21}
only a part of the models can be examined at stage $t$. To obtain the parallel results of Lemmas~\ref{lem:risk1}--\ref{lem:times}, we denote by  $\npull^{(j)}(t):=a_j(t)+b_j(t)-2$ $ (j=1,\ldots,\km)$ the number of times that the $j$th model has been checked, and define  $\hat{\theta}_j (t)=(a_j(t)-1)/(a_j(t)+b_j(t)-2)$ for $j=1,\ldots,\km$ (note that in the situation considered in Theorem \ref{thm:risk22} we have $\hat{\theta}_j (t)=t-1$).
Recall that 
\begin{align*}
\mathcal{A}(t) &= \{M_{\Set_t} \neq M_\true\}  = 
\left\{\xi_{(t)}^{\rm hyb}\neq \rho_t\xi^*_{\true}+(1-\rho_t)\xi_{\rm unif}\right\},\\
 	\mathcal{C}(t) &= \left\{\hat{\theta}_\true (t)-\eta_{\true}(t)\ge \frac{\alpha}{2} \text { or } \exists M_j \neq M_\true \text { s.t. } \eta_j(t)-\hat{\theta}_j(t)\ge \frac{\alpha}{2}\right\},
\end{align*}
and redefine the events
\begin{align*}
   \mathcal{B}(t)= & \Big \{ M_\true \text { can be checked under design } \xi_{(t)}^{\rm hyb}, ~~\hat{\theta}_\true (t)\le  c +\frac{\alpha}{2}\Big \}  \\
   & \bigcup \Big \{ ~\exists M_j \text { can be checked under design } \xi_{(t)}^{\rm hyb},  \text{ and } M_j \neq M_\true \text { s.t. } \hat{\theta}_j(t)\ge  c -\frac{\alpha}{2}\Big \}, \\
   {\mathcal{D}}(t)=&\bigcap_{j\in\{1,\ldots,k\}}\Big \{\npull^{(j)}(t)>\frac{8 \log(T)}{\alpha^2}\Big  \}
\end{align*}
Note that $\npull^{(j)}(t)= {t-1}$ and the models can always be checked for the case we studied in Lemma~\ref{lem:risk1}. Therefore  we still use  notations  
$\mathcal{B}(t)$ and ${\mathcal{D}}(t)$ for the ease of presentation.
According to the construction of $\xi_{(t)}^{\rm hyb}$, one can expect that  the model $M_{\Set_t}$ can be checked for $t=1,\ldots,T$, since $M_{\Set_t}$ is estimable under $\xi_{\Set_t}^*$ and a uniform design provide additional design points for model checking. On the other hand, if a model $M_j$ can not be checked at stage $t$, we have $\xi_{(t)}^{\rm hyb}\neq \rho_t\xi^*_{j}+(1-\rho_t)\xi_{\rm unif}$, i.e., $M_{\Set_t}\neq M_j$ at stage $t$.

\begin{lemma}\label{lem:risk1-remark3}
 Under Condition~\ref{ass:stability1}, we have
	\begin{equation}
		\risk_1(T) :=  \sum_{t=1}^T \E (\mathbb{I}(\mA(t)\cap\mB(t))) 
  \le \km \Big (1+\frac{2}{\alpha^2}\Big ).
	\end{equation}
\end{lemma}

\begin{proof}[Proof of Lemma~\ref{lem:risk1-remark3}]
Note  that 
\begin{align*}
  \risk_1(T) 
  \le & ~\E\Big  [ \sum_{t=1}^T \mathbb{I}\Big (M_\true \text { can be checked under } \xi_{(t)}^{\rm hyb},~\hat{\theta}_{\true}(t) \le  c  + \frac{\alpha}{2} \Big )\Big  ]\\
  &~ + \sum_{M_j\neq M_{\true}}\E\Big  [ \sum_{t=1}^T \mathbb{I}\Big (M_j \text { can be checked under } \xi_{(t)}^{\rm hyb},~\hat{\theta}_{j}(t) \ge  c  - \frac{\alpha}{2} \Big )\Big  ].
\end{align*}
As in  the proof of Lemma~\ref{lem:risk1}, we show 
for the first term
	\begin{equation} \label{det24}
		\E\Big  [ \sum_{t=1}^T \mathbb{I}\left(M_\true \text { can be checked under } \xi_{(t)}^{\rm hyb},~\hat{\theta}_{\true}(t) \le  c  + \frac{\alpha}{2} \right)\Big  ]  \leq 1 + \frac{2}{\alpha^2}
	\end{equation}
 and a similar 
	 inequality for the second term proves the result.  For this purpose  we introduce the notation 
	$$
	\bar\theta_j(t)=\frac{1}{\npull^{(j)}(t)}\sum_{k=1}^t\mathbb{I}(M_j~ \text{is checked at stage}~k)\theta_j(\xi_{(k)}^{\rm hyb}) 
	$$
 and a decomposition of  the time interval $(0,T] = (\iota_0,\iota_1] \cup (\iota_1,\iota_2] \cup \ldots $, where $0= \iota_0 < \iota_1 <  \iota_2 < \ldots$  are  the  times where  model $M_\true $ has been checked. Note that in   each time interval 
 $(\iota_\ell,\iota_{\ell +1}]$ the model  $M_\true $ can only be checked 
 once. 
Recall that the  model $M_\true$ will be checked if it can be checked according to the modification of Algorithm \ref{alg:vanillathompson1} in  Remark  \ref{remark:testset1}.
For $0 < \iota_{w-1} < \iota_w \leq T$ we have
 \begin{align}
  \nonumber   &\E\Big [ \sum_{t=\iota_{w-1}+1}^{\iota_{w} }\mathbb{I}(M_\true \text { can be checked under } \xi_{(t)}^{\rm hyb},~\hat{\theta}_\true (t) <  c +\alpha/2)\Big ] \\
 \nonumber    & ~~= \sum_{\iota_1<\ldots<\iota_{w-1}}\pr\Big (M_\true \text { has been checked at}~\iota_1,\iota_2,\ldots,\iota_w,~\hat{\theta}_\true (\iota_w) <  c +\alpha/2\Big )\\
 \nonumber   & ~~ \le\sum_{\iota_1<\ldots<\iota_{w-1}}\pr\left(\hat{\theta}_\true (\iota_w) - 	\bar\theta_{{\true}}(\iota_w)< -\alpha/2|M_\true \text { has been checked at}~\iota_1,\iota_2,\ldots,\iota_w\right)\\
 \nonumber    &~~~~~~~\times\pr(M_\true \text { has been checked at}~\iota_1,\iota_2,\ldots,\iota_w)\\
 \nonumber    & ~~ \le \exp\left(-w\alpha^2/2\right)\sum_{\iota_1<\ldots<\iota_{w-1}}\pr(M_\true \text { has been checked at}~\iota_1,\iota_2,\ldots,\iota_w)\\
     & ~~ \le  \exp\left(-w\alpha^2/2\right).\label{eq:lem4-1}
 \end{align}
 Here the second inequality is a consequence of Condition~\ref{ass:stability1}
 and  Hoeffding's inequality by observing that 
$$
\bar\theta_{{\true}}(\iota_w)=\E [ \hat\theta_{{\true}}(\iota_w)|M_\true \text { has been checked at}~\iota_1,\iota_2,\ldots,\iota_w ] .
$$
As 
 $\sum_{\iota_1<\ldots<\iota_{w-1}}\pr(M_\true \text { has been checked at}$ $~\iota_1,\iota_2,\ldots,\iota_w)$ equals to  the marginal probability $\pr(M_\true \text { has been checked at}~\iota_w)$, which can not exceed one,
 the last inequality holds as well.

Let $n_\iota \leq T $ denote   the number of times that  model $M_\true$ has been checked, then  
conditioning on the event $\{n_\iota>0\}$ 
and  using \eqref{eq:lem4-1}
yields
	\begin{align*}
		&\E\Big [ \sum_{t=1}^T\mathbb{I}\left(M_\true \text { can be checked under } \xi_{(t)}^{\rm hyb},~
  \hat{\theta}_\true (t) <  c +\alpha/2\right)\Big ] \\
 &  ~~~~ \le 
\E\Big [\sum_{w=0}^{n_\iota-1}\E\Big [ \sum_{t=\iota_{w}+1}^{\iota_{w+1} }\mathbb{I}\left(M_\true \text { can be checked under } \xi_{(t)}^{\rm hyb},~\hat{\theta}_\true (t) <  c +\alpha/2\right)\Big{|}n_{\iota}\ge 1\Big ] \Big  ] \\
  &~~~~~~~~~+\pr(n_{\iota}=0)\\
		 &  ~~~~ \le 
     1 + \sum_{w = 1}^\infty \exp\left(-w\alpha^2/2\right)
   =  1 + \frac{\exp\left(- \alpha^2/2\right)}{1-\exp\left(- \alpha^2/2\right)}\leq 1+ \frac{2}{\alpha^2}~, 
	\end{align*}
which proves \eqref{det24}.
\end{proof}

\begin{lemma} \label{lem:4-remark3} 
	 Under Condition~\ref{ass:stability1} we have  
	$$
	\risk_2(T):=\sum_{t=1}^T\E\{\mathbb{I}(\mathcal{A}(t)\cap \mathcal{B}^c(t)\cap\mathcal{C}(t)\cap \mathcal D(t))\}\leq  \km.
	$$
\end{lemma}
\begin{proof}
	On the event ${\mathcal D}(t)$, we have 
	${\alpha}/{2}>\sqrt{{2\log(T)}/{\npull^{(j)}(t)}}$ for all $j=1,\ldots,\km$. The event $\mathcal{B}^c(t)\cap \mathcal{C}(t)\cap {\mathcal D}(t)$ is a subset of
	$$
	\Big \{ \exists M_j \neq M_\true : \eta_j(t) - \hat{\theta}_j(t) > \sqrt{\frac{2\log(T)}{\npull^{(j)}(t)}} \text{ or }  \hat{\theta}_\true(t)- \eta_\true(t) > \sqrt{\frac{2\log(T)}{\npull^{(j)}(t)}} \Big \}, 
	$$ 
 with $\npull^{(i)}(t)\ge 1$.
	Observing that $\npull^{(j)}(t)=a_j(t)+b_j(t)-2$. 
	and using  Lemma 4 in \cite{wang2018th} we obtain 
	\begin{align*}
		\sum_{t=1}^T \pr\left(\mathcal{A}(t) \cap \mathcal{B}^c(t)\cap \mathcal{C}(t)\cap {\mathcal{D}}(t)\right)&\leq \sum_{t=1}^T \pr\left( \mathcal{B}^c(t)\cap \mathcal{C}(t)\cap {\mathcal{D}}(t)\right) \leq  \km ,
	\end{align*}
which  is the desired result.

\end{proof}

\begin{lemma}\label{lem:times1}
Assume that  Condition~\ref{ass:stability1}
holds  and  that Algorithm~\ref{alg:vanillathompson1} 
is  modified  according to Remark~\ref{remark:testset1}, then
 \begin{equation}
		\E \Big [ \sum_{t=1}^{T} \mathbb{I}\Big  (\mathcal{A}(t)\Big ) \Big ]  
  \le \km\left(3+\frac{4}{\alpha^2}+\frac{8\log(T)}{\alpha^2}\right).
	\end{equation}
\end{lemma}
\begin{proof}[Proof of Lemma~\ref{lem:times1}]
		We use the decomposition 
	\begin{align*}
	\nonumber	\mathbb{I}\left(\mathcal{A}(t)\right) & =   \mathbb{I}\left(\mathcal{A}(t) \cap\mathcal{B}(t)\right) +  \mathbb{I}\left(\mathcal{A}(t) \cap\mathcal{B}^c(t)\cap \mathcal{C}(t) \right) + \mathbb{I}\left(\mathcal{A}(t) \cap\mathcal{B}^c(t)\cap \mathcal{C}^c(t)\right) 
	\end{align*}
 and consider the  three terms separately, where the 
	 first term has already been considered in the proof of Lemma~\ref{lem:risk1-remark3} and the last term is zero.
	
 Thus we focus on the second term, which is  further decomposed 
	\begin{align}
 \mathbb{I}\left(\mathcal{A}(t) \cap\mathcal{B}^c(t)\cap \mathcal{C}(t) \right) &=	
		  \mathbb{I}\left(\mathcal{A}(t) \cap\mathcal{B}^c(t)\cap \mathcal{C}(t) \cap  {\mathcal{D}}(t)  \right) \nonumber\\
		  &+ \mathbb{I}\left(\mathcal{A}(t) \cap\mathcal{B}^c(t)\cap \mathcal{C}(t)\cap   {\mathcal{D}}^c(t)  \right).\label{eq:a21}
	\end{align}
According to Lemma~\ref{lem:4-remark3} the first term is bounded by $\km$, and it remains to 
consider  the second term of \eqref{eq:a21}. 
Note that  $\mathcal{A}(t) \cap\mathcal{B}^c(t)\cap \mathcal{C}(t) \cap  \mathcal{D}^c(t) \subseteq \mathcal{A}(t) \cap\mathcal{D}^c(t)$. 
As only one arm is selected at  time $t$ and $M_j$ can be checked when $\xi^*_j$ is chosen to construct the hybrid design, the 
number of times where $\xi^*_j$ is chosen to construct the hybrid design either exceeds  $8\log(T)/\alpha^2$ times or not.
Observing that  that $N_j(t)\le\npull^{(j)}(t)$, it follows 
$\mathcal{D}^c(t) \subset \{ \exists ~j \text{ such that  } N_j(t)\le 8\log(T)/\alpha^2 \} $.
Thus, for any $M_j\neq M_\true$, the total risk caused by the selection 
is bounded $8\log(T)/\alpha^2$ if  $N_j(t)\le 8\log(T)/\alpha^2$.
Otherwise, if  $N_{j}(t)> 8\log(T)/\alpha^2$,  the  same arguments as given in  the proof of Lemma~\ref{lem:4}
show that the accumulated risk from $t$ to $T$ 
is bounded by $1$.
Therefore,  we have
	\begin{align*}
 \E\left[\mathbb{I}\left(\mathcal{A}(t) \cap\mathcal{B}^c(t)\cap \mathcal{C}(t)\cap  \mathcal{D}^c(t)  \right)\right]\leq 
	& 
 (\km-1)+8(\km-1)\log(T)/\alpha^2 , 
\end{align*}
and the desired result follows. 
\end{proof}

\paragraph{Proof of Theorem~\ref{thm:risk21}.}
The proof now follows by the same arguments as given in the proof of 
Theorem~\ref{thm:risk22}, where 
Lemma \ref{lem:times} is replaced  by Lemma~\ref{lem:times1}.

\section{The designs used in Section~\ref{ex:3}}
\def\theequation{D.\arabic{equation}}
	\setcounter{equation}{0}
 
\label{seca5}
In this subsection, we will give the explicit form of the optimal designs and designs we used for comparison.


Uniform design:
$$
\xi=\begin{Bmatrix}
 0 & 0.25 &0.5 & 0.75&1 \\
  1/5& 1/5 &1/5&1/5&1/5
\end{Bmatrix}.
$$

Standard design:
$$
\xi=\begin{Bmatrix}
 0 & 0.05 &0.2 & 0.6&1 \\
  1/5& 1/5 &1/5&1/5&1/5
\end{Bmatrix}.
$$

Hybrid design:
$$
\xi=\begin{Bmatrix}
 0 & x_{emax}^* &x_{exp}^* &1 \\
  7/18& 2/18 &2/18&7/18
\end{Bmatrix},
$$
where $x_{emax}^*=0.14285$ for both $\delta=3$ and $5$, and $x_{exp}^*=0.67682,0.70560$ 
for $\delta=3$ and $5$, respectively.

Robust design with $\delta=3$ and $5$:
$$
\xi_{\delta=3}=\begin{Bmatrix}
 0 & 0.131 &0.678 &1 \\
  14/46& 9/46 &9/46&14/46
\end{Bmatrix},
$$

$$
\xi_{\delta=5}=\begin{Bmatrix}
 0 & 0.129 &0.726 &1 \\
  14/46& 9/46 &9/46&14/46
\end{Bmatrix},
$$

\end{document}